\documentclass[a4paper,10pt]{article}

\usepackage[utf8]{inputenc}
\usepackage[T1]{fontenc}
\usepackage{graphicx}
\usepackage{amssymb}
\usepackage{amsmath}
\usepackage{amsthm}
\usepackage{enumerate}
\usepackage{vmargin}
\usepackage{tikz}
\usepackage{url}
\usepackage{verbatim}

\usetikzlibrary{matrix,arrows}

\theoremstyle{plain}
\newtheorem{theo}{Theorem}
\newtheorem{lem}{Lemma}
\newtheorem{prop}{Proposition}
\newtheorem{defi}{Definition}
\newtheorem{conj}{Conjecture}

\theoremstyle{remark}
\newtheorem{nota}{Notation}
\newtheorem{rem}{Remark}

\newcommand{\IGl}{\mathrm{IG}\left(2,2 n+1\right)}
\newcommand{\IG}{\mathrm{IG}}

\newcommand{\G}{\mathrm{G}}
\newcommand{\N}{\mathbb{N}}

\newcommand{\ii}{\mathbf{i}}
\newcommand{\dual}[1]{#1^\vee}
\renewcommand{\H}{\mathrm{H}}
\newcommand{\UU}{\mathcal{U}}
\renewcommand{\SS}{\mathcal{S}}
\newcommand{\codim}{\mathrm{codim}\,}
\newcommand{\Hom}[2]{\mathrm{Hom}\left(#1,#2\right)}
\newcommand{\C}{\mathbb{C}}
\newcommand{\Z}{\mathbb{Z}}
\newcommand{\HH}{\mathcal{H}}
\newcommand{\OO}{\mathcal{O}}
\newcommand{\Sp}{\mathrm{Sp}_{2n+1}}
\newcommand{\ra}{\rightarrow}
\newcommand{\SEC}[3]{0\rightarrow #1\rightarrow #2 \rightarrow #3\rightarrow 0}
\newcommand{\T}[2]{\mathrm{T}_{#1}\:#2}

\renewcommand{\O}{\mathbb{O}}
\newcommand{\M}{\overline{\mathcal{M}}}
\newcommand{\Mod}[2]{\overline{\mathcal{M}}_{0,#1}\left(\IG,#2\right)}
\renewcommand{\P}[1]{\mathbb{P}^{#1}}
\newcommand{\YY}{\mathcal{Y}}
\newcommand{\db}[1]{\mathcal{D}^b\left(#1\right)}

\begin{document}

% Titre

\title{Quantum cohomology of the odd symplectic Grassmannian of lines}
\author{Clélia Pech \\ \url{clelia.pech@ujf-grenoble.fr}}
\date{}

\maketitle

% Résumé

\begin{abstract}
Odd symplectic Grassmannians are a generalization of symplectic Grassmannians to odd-dimensional spaces. Here we compute the classical and quantum
cohomology of the odd symplectic Grassmannian of lines. Although these varieties are not homogeneous, we obtain Pieri and Giambelli formulas that 
are very similar to the symplectic case. We notice that their quantum cohomology is semi-simple, which enables us to check Dubrovin's conjecture 
for this case.
\end{abstract}

% Introduction

\section*{Introduction}

The quantum cohomology of homogeneous varieties has been extensively studied (see \cite{tamvakis} for references). Other well-known examples are toric varieties, yet apart from these settings, there are only few examples where the quantum cohomology has been explicitly determined. Quasi-homogeneous varieties provide interesting non toric and non homogeneous examples. Among these two Hilbert schemes have been studied, $\mathrm{Hilb}(2,\P{2})$ \cite{graber} and $\mathrm{Hilb}(2,\P{1}\times\P{1})$ \cite{pontoni}.

In \cite{mihai} Mihai studied a family of varieties, the odd symplectic flag manifolds, which have many features in common with the symplectic flag manifolds. These varieties are interesting at least for two reasons ; first, they are quasi-homogeneous, and secondly, since they have an action of the algebraic group $\Sp$ (the odd symplectic group), whose properties are closely related to those of $\mathrm{Sp}_{2n}$, they are expected to behave almost like homogeneous spaces and thus be relatively easy to deal with. The classical and quantum cohomology of symplectic Grassmannians has been described in \cite{BKT} and \cite{BKT2}, so one can ask whether it is possible to obtain similar results in the case of odd symplectic Grassmannians. 

Here we deal with the case of the odd symplectic Grassmannian of lines $\IGl$, although some of the results about the classical cohomology hold in a more general setting. In \ref{sec:embed.C} and \ref{sec:embed.A} we use the natural embeddings of $\IGl$ in the usual Grassmannian and in the symplectic Grassmannian to compute classical Pieri (see \ref{sec:pieri}) and Giambelli (see \ref{sec:giambelli}) formulas, as well as a presentation of the cohomology ring (see \ref{sec:presentation}). 

For the quantum cohomology the situation is more complicated. Since these varieties are not convex it is necessary to study the moduli spaces corresponding to invariants of degree one to show that they are smooth of the expected dimension. This is done in \ref{sec:mod}. Another difficulty is that since the group action is not transitive, an important transversality result, Kleiman's lemma \cite[Thm. 2]{kleiman} no longer holds. So it will not be possible to force two Schubert varieties to meet transversely by an adequate choice of the defining flags as was done for instance in \cite{coskun}. Hence the Gromov-Witten invariants associated to Schubert varieties are not always enumerative. To solve this problem we replace Schubert varieties by another family of subvarieties and we use a transversality result of Graber \cite{graber} suited for quasi-homogeneous spaces. In \ref{sec:qpieri} we obtain a quantum Pieri formula and a presentation of the quantum cohomology ring. Finally, in \ref{sec:dubrovin}, we check for odd symplectic Grassmannians of lines a conjecture of Dubrovin \cite[Conj. 4.2.2]{dubrovin} relating semisimplicity of the quantum cohomology and the existence of a full exceptional collection in the derived category.

Our results show that there are a many similarities with the symplectic case, since the classical and quantum Pieri formulas are almost the same in both cases. The Hasse diagrams are closely related as well (see \ref{sec:Hasse}). However, Poincar\'e duality is very different, since the Poincar\'e dual of a Schubert class is no longer always a single Schubert class (see \ref{sec:poincare}). Moreover, contrary to what we prove here, the small quantum cohomology ring of the symplectic Grassmannian of lines is not semisimple (see \cite{CP}), and it is not known whether the Dubrovin conjecture holds in this case.

I wish to thank Laurent Manivel for his help on this subject.

% Cohomologie classique

\section{Classical cohomology}

Let $2 \leq m \leq n$ be integers, $V$ be a $\C$-vector space of dimension $2n+1$ and $\omega$ be an antisymmetric form of maximal rank on $V$. We denote its kernel by $K$. The \emph{odd symplectic Grassmannian} is
\[
 \IG_\omega(m,V) := \left\{ \Sigma \in \G(m,V) \mid \Sigma \text{ is isotropic for } \omega \right\}.
\]
It has an action of the \emph{odd symplectic group} :
\[
 \mathrm{Sp}(V) := \left\{ g\in\mathrm{GL}(V) \mid \forall u,v \in V \; \omega(g u,g v) = \omega(u,v) \right\}.
\]
Up to isomorphism, $\IG_\omega(m,V)$ does not depend on the $(2n+1)$-dimensional vector space $V$ nor on the form $\omega$, so we may denote it by $\IG(m,2n+1)$. Similarly, from now on we denote $\mathrm{Sp}(V)$ by $\Sp$. We recall some basic facts from \cite[Prop. 4.1 and 4.3]{mihai} :

\begin{prop}
\label{prop:mihai.1}
 \begin{enumerate}
  \item The odd symplectic Grassmannian $\IG(m,2n+1)$ is a smooth subvariety of codimension $\frac{m(m-1)}{2}$ of the usual 
Grassmannian $\G(m,2n+1)$.
  \item Moreover, it has two orbits under the action of the odd symplectic group $\Sp$ :
\begin{itemize}
 \item the closed orbit $\O := \left\{ \Sigma \in \IG(m,2n+1) \mid \Sigma \supset K \right\}$, which is isomorphic to the symplectic Grassmannian
$\IG(m-1,2n)$ ;
 \item the open orbit $\left\{ \Sigma \in \IG(m,2n+1) \mid \Sigma \not \supset K \right\}$, which is isomorphic to the dual of the tautological 
bundle over the symplectic Grassmannian $\IG(m,2n)$.
\end{itemize}
\end{enumerate}
\end{prop}

For us, a \emph{quasi-homogeneous space} will be an algebraic variety endowed with an action of an algebraic group with only finitely many orbits. Odd symplectic Grassmannians are examples of such spaces. 

\subsection{Schubert varieties}

A $\C$-vector space $V$ of dimension $2n+1$ endowed with an antisymmetric form of maximal rank $\omega$ can be embedded in a symplectic space $(\overline{V},\overline{\omega})$ of dimension $2n+2$ such that $\overline{\omega}\mid_V = \omega$. This construction gives rise to a natural embedding $\ii : \IG(m,2n+1) \hookrightarrow \IG(m,2n+2)$. It can be easily seen that $\ii$ identifies $\IG(m,2n+1)$ with a Schubert subvariety of $\IG(m,2n+2)$. Moreover this embedding enables us to obtain a description of the Schubert subvarieties of $\IG(m,2n+1)$. In \ref{sec:schub.even} we recall some facts about Schubert varieties in $\IG(m,2n)$, then in \ref{sec:schub.odd} we describe the Schubert varieties of $\IG(m,2n+1)$ and introduce an indexation using partitions. 

\subsubsection{Schubert varieties in the symplectic Grassmannian}
\label{sec:schub.even}

Here we recall the indexing conventions introduced in \cite[Def. 1.1]{BKT}. Two kinds of combinatorial objects can be used to index Schubert varieties of the symplectic Grassmannian $\IG(m,2n)$, $k$-strict partitions (with $k:=n-m$) and index sets :
\begin{defi}
\begin{enumerate}
 \item A \emph{$k$-strict partition} is a weakly decreasing sequence of integers $\lambda = (\lambda_1\geq\dots\geq\lambda_{m}\geq 0)$ such that $\lambda_j > k \Rightarrow \lambda_j > \lambda_{j+1}$. 
 \item An \emph{index set of length $m$} for the symplectic Grassmannian is a subset $P=(p_1<\dots<p_m) \subset \left[ 1 , 2n \right]$ with $m$ elements such that for all $1\leq i,j\leq m$ we have $p_i+p_j \neq 2n+1$.
\end{enumerate}
\end{defi}
Now if $F_{\bullet}$ is an \emph{isotropic flag} (i.e a complete flag such that $F_{n-i}^\perp = F_{n+i}$ for all $0\leq i\leq n$), to each admissible index set $P = (p_1,\dots,p_m)$ of length $m$  we can associate the \emph{Schubert cell}
\[
 X^{\circ}_P(F_\bullet) := \left\{ \Sigma \in \IG(m,2n)\mid\dim(\Sigma\cap F_{p_j})=j,\;\forall\; 1\leq j\leq m\right\}.
\]
Moreover there is a bijection between $k$-strict partitions $\lambda$ such that $\lambda_1\leq 2n-m$ and index sets $P\subset \left[1,2n\right]$ of length $m$, given by
\begin{align*}
 \lambda&\mapsto P=(p_1,\dots,p_{m}) \text{ where $p_j=n+k+1-\lambda_j+\#\left\{i<j\mid\lambda_i+\lambda_j\leq 2k+j-i\right\}$,} \\
 P&\mapsto\lambda=(\lambda_1,\dots,\lambda_{m}) \text{ where $\lambda_j=n+k+1-p_j+\#\left\{i<j\mid p_i+p_j>2n+1\right\}$}.
\end{align*}
The advantage of the representation by $k$-strict partitions is twofold : it mimics the indexation of Schubert classes of type A Grassmannians by 
partitions, and the codimension of the Schubert variety associated to a $k$-strict partition $\lambda$ is easily computed as $|\lambda|=\sum_{j=1}^m\lambda_j$. In the next paragraph we will describe a similar indexation for the odd symplectic Grassmannian.

\subsubsection{Schubert varieties in the odd symplectic Grassmannian}
\label{sec:schub.odd}

We now use Mihai's description of the odd symplectic Grassmannian as a Schubert subvariety of $\IG(m,2n+2)$ to define the Schubert varieties of the odd symplectic Grassmannian. We also introduce two indexations for them.

Schubert varieties of the odd symplectic Grassmannian will be defined with respect to an \emph{isotropic flag of $\C^{2n+1}$}, i.e a complete flag $F_\bullet$ which is the restriction of an isotropic flag $F_\bullet^+$ of $\C^{2n+2}$. Denote by $1^m$ the partition $\lambda^0$ such that $\lambda_1^0=\dots=\lambda^0_{m}=1$. It corresponds to the index set $P^0=(2n+2-m,\dots,2n+1)$. 

\begin{prop}
\label{prop:mihai.2}
The embedding $\ii :\IG(m,2n+1)\ra\IG(m,2n+2)$ identifies $\IG(m,2n+1)$ with the Schubert subvariety of $\IG(m,2n+2)$ associated to the $(n+1-m)$-strict partition $\lambda^0$ (or, equivalently, to the index set $P^0$). 
\end{prop}

We define the Schubert varieties of $\IG(m,2n+1)$ as the subvarieties of $\IG(m,2n+1)$ of the form
\[
 X_P(F_\bullet) := \left\{ \Sigma \in \IG(m,2n+1) \mid \dim (\Sigma \cap F_{p_j}) \geq j \text{ for all $j$} \right\},
\]
where 
\begin{itemize}
 \item $P$ is an index set of length $m$ of $\left[1,2n+1\right]$, that is, a $m$-uple $P=(p_1\leq\dots\leq p_m)$ with $1\leq p_j\leq 2n+1$ for all $j$ and $p_i+p_j\neq 2n+3$ for all $i,j$ ;
 \item $F_\bullet$ is an isotropic flag of $\C^{2n+1}$.
\end{itemize}
These varieties coincide with the Schubert varieties of $\IG(m,2n+2)$ indexed by index sets $P$ of $\left[1,2n+2\right]$ such that $P\leq P^0$ (for the lexicographical order), hence Proposition \ref{prop:mihai.2} implies that they define a cellular decomposition on $\IG(m,2n+1)$.

Let us now describe another indexation of the Schubert varieties of $\IG(m,2n+1)$ using partitions. If $P$ is an index set of $\left[1,2n+1\right]$, we associate to it a $(n-m)$-strict $m$-uple of weakly decreasing integers $\lambda=(\lambda_1\geq\dots\geq\lambda_m\geq -1)$ defined by
\[
 \lambda_j=2n+2-m-p_j+\#\left\{i<j\mid p_i+p_j>2n+3\right\} \text{ for all $1\leq j\leq m$.}
\]
Conversely if $\lambda=(\lambda_1\geq\dots\geq\lambda_m\geq -1)$ is any $(n-m)$-strict $m$-uple of weakly decreasing integers such that $\lambda_1\leq 2n+1-m$ and $(\lambda_m=-1 \Rightarrow \lambda_1=2n+1-m)$, then the assignement
\[
 p_j=2n+2-m-\lambda_j+\#\left\{i<j\mid \lambda_i+\lambda_j\leq 2(n-m)+j-i\right\} \text{ for all $1\leq j\leq m$}
\]
defines an index set of $\left[1,2n+1\right]$. It is easy to check that with respect to this indexation convention, the Schubert variety $X_{\lambda}(F_{\bullet})$ has codimension $|\lambda|$ in $\IG(m,2n+1)$. 

\begin{rem}
 For the case of the \emph{odd symplectic Grassmannian of lines} $\IGl$, it follows that the indexing partitions can be either
 \begin{itemize}
  \item ``usual'' $(n-2)$-strict partitions $\lambda=(2n-1\geq\lambda_1\geq\lambda_2\geq 0)$ ;
  \item the ``partition'' $\lambda=(2n-1,-1)$ corresponding to the class of the closed orbit $\O$.
 \end{itemize}
\end{rem}

\subsection{Embedding in the symplectic Grassmannian}
\label{sec:embed.C}

Now we draw some consequences of the embedding of $\IG(2,2n+1)$ as a Schubert subvariety of a symplectic Grassmannian. Since we know the cohomology of $\IG(2,2n+2)$, describing the restriction map $\ii^*$ will give us information on the cohomology of $\IGl$. 

Let $F_\bullet$ be an isotropic flag, $Y_{a,b}(F_\bullet)$ a Schubert subvariety of $\IG(2,2 n+2)$ and $\upsilon_{a,b}$ the associated Schubert class, where $(a,b)$ is an $(n-2)$-strict partition. From Proposition \ref{prop:mihai.2}, we know that $\IGl$ is isomorphic to the Schubert subvariety $Y_{1,1}(E_\bullet)$ of $\IG(2,2n+2)$, where $E_\bullet$ is an isotropic flag which we may assume to be in general position with respect to $F_\bullet$. Then it follows that $Y_{a,b}(F_\bullet)$ and $Y_{1,1}(E_\bullet)$ meet transversally, hence we can compute the restriction $\ii^*\upsilon_{a,b}$ by computing the class of the intersection $Y_{a,b}\cap Y_{1,1}$ in $\IG(2,2 n+2)$ using the classical Pieri rules for $\IG(2,2n+2)$ \cite[Thm 1.1]{BKT} :
\begin{align*}
 \upsilon_{a,b}\cup\upsilon_{1,1}=\begin{cases}
                                    \upsilon_{a+1,b+1} & \text{if $a+b\neq 2n-2,2n-1$,} \\
				    \upsilon_{a+1,b+1}+\upsilon_{a+2,b} & \text{if $a+b=2n-2$ or $2n-1$.}
                                   \end{cases}
\end{align*}
\begin{rem}
 In the above formula, we should remove classes that are not indexed either by $(n-2)$-strict partitions $\lambda=(2n-1\geq\lambda_1\geq\lambda_2\geq 0)$ or by the special partition $\lambda=(2n-1,-1)$. We will adopt this convention throughout the rest of the text to simplify formulas.
\end{rem}

Denote by $\tau_\lambda$ the cohomology class associated to the Schubert variety $X_\lambda(F_\bullet)\subset\IGl$, where $\lambda$ is a $(n-2)$-strict partition and $F_\bullet$ is an isotropic flag of $\C^{2n+1}$. This class does not depend on the choice of the isotropic flag. 

Looking at the incidence conditions for the corresponding Schubert varieties, we prove that for each $\tau_{c,d}\in\H^*(\IGl,\Z)$, we have
\[
 \ii_*\tau_{c,d}=\upsilon_{c+1,d+1}.
\]
Using the projection formula $\ii_*(\alpha\cup\ii^*\beta)=\ii_*\alpha\cup\beta$, we deduce the
\begin{lem}[Restriction formula]
\label{lem:restriction}
 Let $\upsilon_{a,b}\in\H^*(\IG(2,2n+2),\Z)$ be a Schubert class. Then its restriction to the odd symplectic Grassmannian $\IGl$ is given by
 \begin{align*}
  \ii^*\upsilon_{a,b}=\begin{cases}
		       \tau_{a,b} & \text{if $a+b\neq 2n-2,2n-1$,} \\
		       \tau_{a,b}+\tau_{a+1,b-1} & \text{if $a+b=2n-2$ or $2n-1$.}
		      \end{cases}
 \end{align*}
\end{lem}
In particular we notice that $\ii^*$ is surjective and has kernel generated by the class $\upsilon_{2n}$. 
So the classical cohomology of $\IGl$ is entirely determined by the classical cohomology of $\IG\left(2,2 n+2\right)$. 

\subsection{Poincar\'e duality}
\label{sec:poincare}

If $X$ is a smooth algebraic variety and $(\gamma_i)$ a homogeneous basis of its cohomology ring, we denote by $(\dual{\gamma}_i)$ the corresponding Poincaré dual basis. For homogeneous spaces and for odd symplectic Grassmannians, the basis $(\gamma_i)$ we consider is the basis of Schubert classes. Here we compute Poincar\'e duality for $\IGl$.

If $\alpha=\upsilon_{a,b}$ is a Schubert class such that $b\geq 1$ or $(a,b)=(2n,0)$, then there exists a unique class $\gamma$ in $\IGl$ such that $\ii_*\gamma = \alpha$. We denote it by $\alpha_{-}$. 
We first prove the 
\begin{lem}
\label{lem:poincare}
 Let $\alpha=\upsilon_{a,b}$ be a cohomology class in $\IG(2,2n+2)$ such that $b\geq 1$ or $(a,b)=(2n,0)$. Then $\ii^*\dual{\alpha}=\dual{\alpha_{-}}$. 
\end{lem}
\begin{proof}[Proof of the lemma]
By definition of Poincar\'e duality, if $\alpha$ and $\beta$ are two cohomology classes in $\IG(2,2n+2)$, then 
\[
 \int_{\IG(2,2 n+2)}\alpha\cup\dual{\beta}=\delta_{\alpha,\beta},
\]
where $\delta$ is the Kronecker symbol.
So 
\begin{equation}
\label{eq:poincare}
 \int_{\IG(2,2n+2)}(\ii_*\alpha_{-})\cup\dual{\beta}=\delta_{\alpha,\beta}=
 \int_{\IG(2,2n+2)}\ii_*(\alpha_{-}\cup\ii^*\dual{\beta}).
\end{equation}
Expressing $\ii^*\dual{\beta}$ on the dual basis in $\IGl$, we get $\ii^*\dual{\beta}=\sum_\gamma x_{\beta,\gamma}\dual{\gamma}$. 
Hence
\[
 \delta_{\alpha,\beta}=\sum_\gamma x_{\beta,\gamma}\int_{\IG(2,2n+2)}\ii_*(\alpha_{-}\cup\dual{\gamma})=
\sum_\gamma x_{\beta,\gamma}\delta_{\alpha_{-},\gamma}.
\]
So $x_{\beta,\alpha_{-}}=\delta_{\alpha,\beta}$, and the result follows.
\end{proof}
Finally, Poincar\'e duality in $\IGl$ takes the following form :
\begin{prop}[Poincar\'e duality]
\label{prop:poincare}
 \begin{align*}
 \dual{\tau_{a,b}}=\begin{cases}
        \tau_{2n-1-b,2n-2-a}  & \text{if $a+b<2n-2$,} \\
	\tau_{2n-2-b,2n-1-a}+\tau_{2n-1-b,2n-2-a} & \text{if $a+b=2n-2$ or $2n-1$,} \\
        \tau_{2n-2-b,2n-1-a} & \text{if $a+b>2n-1$.}
     \end{cases}
\end{align*}
\end{prop}

\begin{proof}
 We will derive this result from Poincar\'e duality on $\IG(2,2n+2)$ using Lemmas \ref{lem:restriction} and \ref{lem:poincare}. Indeed, we prove with the projection formula that if $\alpha$ is a class in $\IG(2,2n+2)$, then $\dual{\alpha_{-}}=(\dual{\alpha}\cup\upsilon_{1,1})_{-}$. Then using the Poincar\'e duality formula in $\IG(2,2n+2)$ proved in \cite[§ 4.1]{BKT}, an easy calculation gives the result.
\end{proof}

\begin{rem}
 This result is very different from what we get for the usual Grassmannians or even the symplectic or orthogonal ones. 
Indeed, the basis of Schubert classes is not self-dual. This fact will have many consequences ; in particular, the Hasse diagram
of $\IGl$ (see Figure \ref{fig:hasse.2.7}) will be much less symmetric that the Hasse diagram of, say, $\IG(2,2n+2)$ (see 
Figure \ref{fig:hasse.2.6}).
\end{rem}

\subsection{Pieri formula}
\label{sec:pieri}

To compute the cup product of two cohomology classes in $\IGl$, we need two ingredients : a \emph{Pieri formula} describing the cup product
of any Schubert class with a special class (that is, one of the classes $\tau_{1}$ or $\tau_{1,1}$), and a \emph{Giambelli formula} decomposing any
Schubert class as a polynomial in $\tau_1$ and $\tau_{1,1}$. In this paragraph we describe the Pieri formula as well as an alternative rule for
multiplying Schubert classes and classes of the form $\tau_p$ with $0\leq p\leq 2n-1$ or $\tau_{2n-1,-1}$.

We start by expressing cohomology classes in $\IGl$ in terms of cohomology classes in $\IG(2,2n+2)$ using Lemma \ref{lem:restriction} :
 \begin{align*}
 \tau_{c,d}=\begin{cases}
        \mathbf{i}^*\upsilon_{c,d} & \text{if $c+d\neq 2n-2,2n-1$,} \\
        \sum_{j=0}^{c-n}(-1)^{c-n-j}\mathbf{i}^*\upsilon_{n-1+j,n-1-j}  & \text{if $c+d=2n-2$,} \\
	\sum_{j=c-n}^{n-1}(-1)^{j-c+n}\mathbf{i}^*\upsilon_{n+j,n-1-j}  & \text{if $c+d=2n-1$}.
     \end{cases}
\end{align*}
Now combining this with the Pieri rule in $\IG(2,2n+2)$, we can prove a Pieri rule for $\IGl$~:
\begin{prop}[Pieri formula]
\label{prop:pieri.1}
 \begin{align*}
  \tau_{a,b}\cup\tau_1 & =\begin{cases}
                        \tau_{a+1,b}+\tau_{a,b+1} & \text{if $a+b\neq 2n-3$,} \\
			\tau_{a,b+1}+2\tau_{a+1,b}+\tau_{a+2,b-1} & \text{if $a+b=2n-3$.}
                        \end{cases} \\
  \tau_{a,b}\cup\tau_{1,1}& =\begin{cases}
                        \tau_{a+1,b+1} & \text{if $a+b\neq 2n-4,2n-3$,} \\
			\tau_{a+1,b+1}+\tau_{a+2,b} & \text{if $a+b=2n-4$ or $2n-3$.}
                        \end{cases}
 \end{align*}
\end{prop}
We may also state a rule for multiplying by the Chern classes of the quotient bundle :
\begin{align*}
 c_p(\mathcal{Q})=\begin{cases}
                   \tau_p & \text{if $0\leq p\leq 2n-1$ and $p\neq 2n-2$,} \\
		   \tau_{2n-2}+\tau_{2n-1,-1} & \text{if $p=2n-2$}.
                  \end{cases}
\end{align*}
We prove in the same way as Proposition \ref{prop:pieri.1} the
\begin{prop}[another Pieri formula]
\label{prop:pieri.2}
  \begin{align*}
 \tau_{a,b}\cup\tau_p=\begin{cases}
        \displaystyle\sum_{\substack{(a+1,b+1) \ra (c+1,d+1) \\ c+d=a+b+p \\ d\geq 0 \text{ or } c=2n-1}} 2^{N\left((a+1,b+1),(c+1,d+1) \right)} \tau_{c,d},  \\ 
	\hspace{3.8cm} \text{if $p\neq 2n-2$ or $(a+b\neq 2n-1$ and $(a,b)\neq (2n-1,-1))$,} \\
        (-1)^a\tau_{2n-1,2n-2}  \hspace{1.3cm} \text{if $p=2n-2$, $a+b=2n-1$ and $b\neq 0$,} \\
	0 \hspace{3.6cm} \text{if $p=2n-2$ and $((a,b)=(2n-1,-1)$ or $(2n-1,0))$},
     \end{cases}
\end{align*}
 where the relation $\lambda \ra \mu$ and the integer $N(\lambda,\mu)$ are defined in \cite[Def 1.3]{BKT}.
  \begin{align*}
 \hspace{-5cm}\tau_{a,b}\cup\tau_{2n-1,-1}=\begin{cases}
	(-1)^{a-1}\tau_{2n-1,2n-2} & \text{if $a+b=2n-1$,} \\
	\tau_{2n-1,a-1} & \text{if $b=0$ and $a\neq  2n-2$,} \\
	\tau_{2n-1,2n-3} & \text{if $(a,b)=(2n-1,-1)$,} \\
	0 & \text{else}.
     \end{cases}
\end{align*}
\end{prop}

Notice that contrary to the symplectic case (and to the case of other homogeneous spaces) we sometimes get negative coefficients for the second 
Pieri rule. It is a consequence of the fact that we only have a quasi-homogeneous space, so it is not always possible to find representatives of 
the two Schubert varieties that intersect transversally. So even in degree $0$ Gromov-Witten invariants associated to Schubert classes are not always enumerative, contrary to the case of homogeneous spaces. That is why we have to outline conditions in \ref{sec:trans} to recover enumerativity for some invariants.

\subsection{The Hasse diagram of $\IGl$}
\label{sec:Hasse}

The Pieri rule from Proposition \ref{prop:pieri.1} enables us in particular to compute the multiplication by the hyperplane class $\tau_1$. The 
corresponding graph is called the \emph{Hasse diagram} of $\IGl$. More precisely, the Hasse diagram of $\IGl$ is an oriented graph with multiplicity such that :
\begin{itemize}
 \item its vertices are the Schubert classes of $\IGl$ ;
 \item two vertices $\tau_{a,b}$ and $\tau_{c,d}$ are related by an arrow of multiplicity $r$ if $\tau_{c,d}$ appears with multiplicity $r$ in the product $\tau_{a,b}\cup\tau_1$.
\end{itemize}
For instance see Figure \ref{fig:hasse.2.7} for the Hasse diagram of $\IG(2,7)$. Arrows are going from left to right.
\begin{figure}[h!]
 \centering
 \begin{tikzpicture}[scale=1.5]
\tikzstyle{every node}=[draw,circle,fill=white,minimum size=6pt,inner sep=0pt]

\draw[red] (0,0) node (0) [label=below:$\tau_{\emptyset}$] {}
        -- (1,0) node (1) [label=below:$\tau_1$] {}
	-- (2,0.5) node (11) [label=above:$\tau_{1,1}$] {}
	-- (3,0.5) node (21) [label=above:$\tau_{2,1}$] {}
	-- (4,0) node (4) [label=above:$\tau_{4}$] {}
	-- (5,0) node (41) [label=below:$\tau_{4,1}$] {}
	-- (6,0.5) node (42) [label=above:$\tau_{4,2}$] {}
	-- (7,0.5) node (43) [label=above:$\tau_{4,3}$] {};
\draw[red] (1)
        -- (2,-0.5) node (2) [label=below:$\tau_{2}$] {}
        -- (3,-0.5) node (3) [label=below:$\tau_{3}$] {}
	-- (4,1) node (31) [label=above:$\tau_{3,1}$] {}
	-- (5,1) node (32) [label=above:$\tau_{3,2}$] {}
	-- (42);
\draw[red]
	(2) -- (21)
	(31) -- (41);
\draw[red,double]
	(21) --(31)
	(3) -- (4);
\node[draw=none,text=red] at (2,1.3) {$\mathrm{IG}(2,6)$};

\draw[blue,label distance=-5mm] (4,-1) node (5-1) [label=below:{$\mathbb{O}=\tau_{5,-1}$}] {};
\draw[blue] (5-1)
       -- (5,-1) node (5) [label=below:$\tau_{5}$] {}
        -- (6,-0.5) node (51) [label=below:$\tau_{5,1}$] {}
        -- (7,-0.5) node (52) [label=below:$\tau_{5,2}$] {}
        -- (8,0) node (53) [label=below:$\tau_{5,3}$] {}
        -- (9,0) node (54) [label=below:$\tau_{5,4}$] {};
\draw 
	(3) -- (5-1)
	(4) -- (5)
	(41) -- (51)
	(42) -- (52)
	(43) -- (53);
\node[draw=none,text=blue] at (6,-1.3) {$\mathrm{IG}(1,6)$};

\end{tikzpicture}

 \caption{Hasse diagram of $\IG(2,7)$}
 \label{fig:hasse.2.7}
\end{figure}
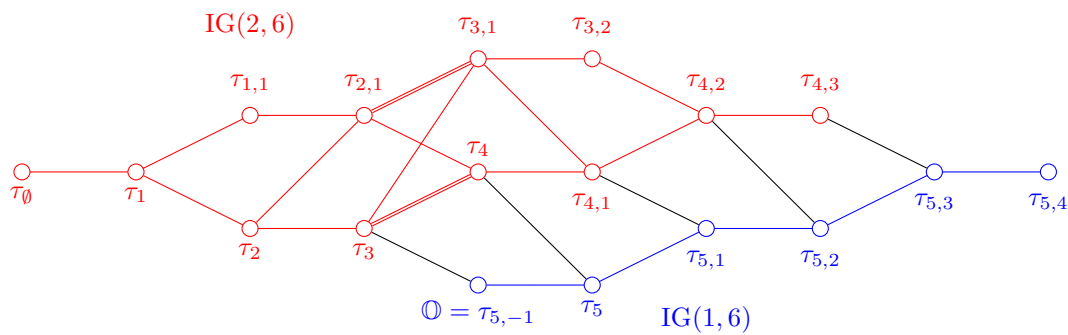

As a comparison, see also the Hasse diagram of the symplectic Grassmannian $\IG(2,6)$ in Figure \ref{fig:hasse.2.6}, and of 
$\IG(2,8)$ in Figure \ref{fig:hasse.2.8}.
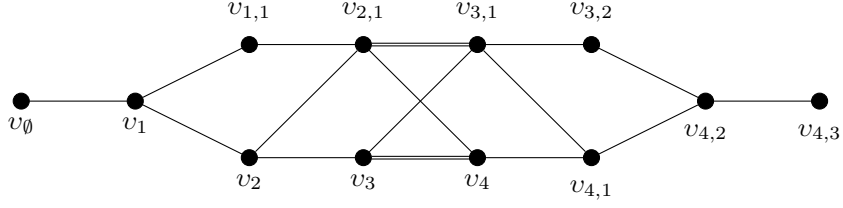
\begin{figure}
 \centering
 \begin{tikzpicture}[scale=1.5]
 
\tikzstyle{every node}=[draw,circle,fill=black,minimum size=6pt,inner sep=0pt]

\draw (0,0) node (0) [label=below:$\upsilon_{\emptyset}$] {}
        -- (1,0) node (1) [label=below:$\upsilon_1$] {}
        -- (2,-0.5) node (2) [label=below:$\upsilon_{2}$] {}
        -- (3,-0.5) node (3) [label=below:$\upsilon_{3}$] {}
        -- (4,0.5) node (31) [label=above:$\upsilon_{3,1}$] {}
        -- (5,0.5) node (32) [label=above:$\upsilon_{3,2}$] {}
        -- (6,0) node (42) [label=below:$\upsilon_{4,2}$] {}
        -- (7,0) node (43) [label=below:$\upsilon_{4,3}$] {};
\draw (1) 
	-- (2,0.5) node (11) [label=above:$\upsilon_{1,1}$] {}
	-- (3,0.5) node (21) [label=above:$\upsilon_{2,1}$] {}
	-- (4,-0.5) node (4) [label=below:$\upsilon_{4}$] {}
	-- (5,-0.5) node (41) [label=below:$\upsilon_{4,1}$] {}
	-- (42);
\draw[double] 
	(21) -- (31)
	(3) -- (4);
\draw 
	(2) --(21)
	(31) -- (41);

\end{tikzpicture}

 \caption{Hasse diagram of $\IG(2,6)$}
 \label{fig:hasse.2.6}
\end{figure}
\begin{figure}
 \centering
 \begin{tikzpicture}[scale=1.3]
 
\tikzstyle{every node}=[draw,circle,fill=white,minimum size=6pt,inner sep=0pt]

\draw[blue] (0,0) node (0) [label=below:$\upsilon_{\emptyset}$] {}
        -- (1,0) node (1) [label=below:$\upsilon_1$] {}
	-- (2,0.5) node (2) [label=above:$\upsilon_{2}$] {}
	-- (3,0.5) node (3) [label=above:$\upsilon_{3}$] {}
	-- (4,1) node (4) [label=above:$\upsilon_{4}$] {}
	-- (5,1) node (5) [label=above:$\upsilon_{5}$] {};
\node[draw=none,text=blue] at (3,1.3) {$\mathrm{IG}(1,6)$};
\draw[red] (2,-0.5) node (11) [label=below:$\upsilon_{1,1}$] {}
	-- (3,-0.5) node (21) [label=below:$\upsilon_{2,1}$] {}
	-- (4,0) node (31) [label=above:$\upsilon_{3,1}$] {}
	-- (5,0) node (41) [label=above:$\upsilon_{4,1}$] {}
	-- (6,1) node (6) [label=above:$\upsilon_{6}$] {}
	-- (7,1) node (61) [label=above:$\upsilon_{6,1}$] {}
	-- (8,0.5) node (62) [label=above:$\upsilon_{6,2}$] {}
	-- (9,0.5) node (63) [label=above:$\upsilon_{6,3}$] {}
	-- (10,0) node (64) [label=above:$\upsilon_{6,4}$] {}
	-- (11,0) node (65) [label=above:$\upsilon_{6,5}$] {};
\draw[red] (21)
	-- (4,-1) node (22) [label=below:$\upsilon_{2,2}$] {}
	-- (5,-1) node (32) [label=below:$\upsilon_{3,2}$] {}
	-- (6,0) node (51) [label=above:$\upsilon_{5,1}$] {}
        -- (7,0) node (52) [label=above:$\upsilon_{5,2}$] {}
        -- (8,-0.5) node (53) [label=below:$\upsilon_{5,3}$] {}
        -- (9,-0.5) node (54) [label=below:$\upsilon_{5,4}$] {}
	-- (64);
\draw[red] (41)
	-- (6,-1) node (42) [label=below:$\upsilon_{4,2}$] {}
	-- (7,-1) node (43) [label=below:$\upsilon_{4,3}$] {}
	-- (53);
\draw[red] 
	(31) -- (32)
	(51) -- (61)
	(52) -- (62)
	(53) -- (63)
	(42) -- (52);
\draw[red,double]
	(41) -- (51)
	(32) -- (42);
\node[draw=none,text=red] at (8,-1.5) {$\mathrm{IG}(2,7)$};
\draw
	(1) -- (11)
	(2) -- (21)
	(3) -- (31)
	(4) -- (41)
	(5) -- (51);
\draw[double]
	(5) -- (6);
\end{tikzpicture}

 \caption{Hasse diagram of $\IG(2,8)$}
 \label{fig:hasse.2.8}
\end{figure}
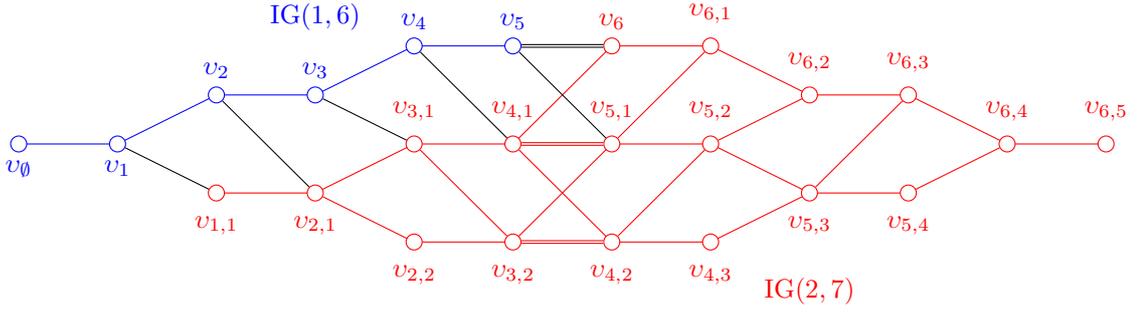

Looking at these examples we notice that the Hasse diagram of  $\IG(2,7)$ contains the Hasse diagram of $\IG(2,6)$ as a subgraph, 
the subgraph induced by the remaining vertices being isomorphic to the Hasse diagram of $\IG(1,6)$. Moreover, the Hasse diagram of  
$\IG(2,8)$ contains the Hasse diagram of $\IG(2,7)$ as a subgraph, the subgraph induced by the remaining vertices being isomorphic 
to the Hasse diagram of $\IG(1,6)$.
This is a general fact.
More precisely, we have the following decomposition of the Hasse diagrams of the even and odd symplectic Grassmannian :
\begin{prop}
\begin{itemize}
 \item  The Hasse diagram of $\IG(2,2n+1)$ is isomorphic to the disjoint union of :
\begin{enumerate}
 \item the Hasse diagram of $\IG(2,2n)$, whose vertices are the classes in $\IG(2,2n+1)$ associated to the Schubert varieties not 
contained in the closed orbit ;
\item the Hasse diagram of the closed orbit $\O\cong\IG(1,2n)$ ;
\end{enumerate}
with parts $1$ and $2$ linked by the simple arrows joining $\tau_{2n-3}$ to $\tau_{2n-1,-1}$ and $\tau_{2n-2,a}$ to $\tau_{2n-1,a}$ for 
$0\leq a\leq 2n-3$.
\item The Hasse diagram of $\IG(2,2n)$ is isomorphic to the disjoint union of :
\begin{enumerate}
 \item the Hasse diagram of $\IG(2,2n-1)$, whose vertices are the classes in $\IG(2,2n)$ associated to the Schubert varieties 
contained in $X_{1,1}$ ;
\item the Hasse diagram of $\IG(1,2n-2)$, corresponding to the classes $\tau_{\emptyset}$ to $\tau_{2n-3}$ ;
\end{enumerate}
with parts $1$ and $2$ linked by the double arrow joining $\tau_{2n-3}$ to $\tau_{2n-2}$ and the simple arrows joining $\tau_{p}$ to $\tau_{p,1}$ 
for $1\leq p\leq 2n-3$.
\end{itemize}
\end{prop}

\begin{proof}
We will denote by $\HH_{\IG(m,N)}$ the Hasse diagram of $\IG(m,N)$. 
\begin{itemize}
 \item Let $G_1$ be the subgraph of $\HH_{\IG(2,2n+1)}$ induced by the vertices $\tau_\lambda$ for $\lambda$ such that $\lambda_1<2n-1$. 
We need to prove that $G_1=\HH_{\IG(2,2n)}$. First notice these graphs have the same set of vertices. Then use the diagram :

\begin{center}
  \begin{tikzpicture}
  \matrix [matrix of math nodes,row sep=1cm, column sep=1cm,text height=1.5ex,text depth=0.25ex]
  {
	 |(O)| \mathbb{O} & |(IGp)| \mathrm{IG}(2,2n) \\
	 |(IG)| \mathrm{IG}(2,2n+1) & \\
};
\path[right hook->]
(O) edge node[auto] {$i$} (IG)
;
\path[->]
(O) edge node[auto] {$\phi$} (IGp)
;
\end{tikzpicture}

\end{center}

where $i$ is the natural inclusion and $\phi(\Sigma)=\Sigma/K$ for each $\Sigma\in\O$. Looking at incidence conditions we notice that 
$\phi^*\upsilon_\lambda=i^*\tau_\lambda$ for each Schubert class $\upsilon_\lambda$ of $\IG(2,2n)$, and we get
\[
 \phi^*\left(\upsilon_1\cup\upsilon_\lambda\right)=\phi^*\upsilon_1\cup\phi^*\upsilon_\lambda=i^*(\tau_1\cup\tau_\lambda),
\]
hence $G_1$ and $\HH_{\IG(2,2n)}$ have the same arrows. 
Now the vertices of $\HH_{\IG(2,2n+1)}$ not contained in $G_1$ correspond to the classes $\tau_\lambda$ with $\lambda_1=2n-1$, that is to 
the Schubert varieties contained in the closed orbit $\O\cong \P{2n-1}$. So the graph $G_2$ they induce is isomorphic to $\IG(1,2n)$. 
Finally, the arrows joining $G_1$ and $G_2$ are determined using the Pieri rule \ref{prop:pieri.1}.
 \item For $\IG(2,2n)$ the result is simply a consequence of the isomorphism between $\IGl$ and the Schubert subvariety $X_{1,1}$ of 
$\IG(2,2n)$ stated in \ref{sec:embed.C}, and of the Pieri rule for $\IG(2,2n)$ proved in \cite[Thm. 1.1]{BKT}. \qedhere
\end{itemize}
\end{proof}

 This result can be easily generalized to all symplectic Grassmannians $\IG(m,N)$ :
\begin{prop}
 \begin{itemize}
  \item The Hasse diagram of $\IG(m,2n)$ is made of the union of :
    \begin{enumerate}
     \item the Hasse diagram $\HH_1$ of $\IG(m,2n-1)$, whose vertices are the cohomology classes of $\IG(m,2n)$ associated to the Schubert varieties contained in $X_{1^m}$ ;
     \item the Hasse diagram $\HH_2$ of $\IG(m-1,2n-2)$.
    \end{enumerate}
    The arrows from $\HH_2$ to $\HH_1$ are of two types :
    \begin{itemize}
     \item simple arrows $\lambda \ra \mu$ for $\lambda,\mu$ such that $\lambda_1\leq 2n-1-m$, $\lambda_{m-1}\geq 1$, $\lambda_m=0$ and $\mu=(\lambda_1,\dots,\lambda_{m-1},1)$ ;
     \item double arrows $\lambda \ra \mu$ for $\lambda,\mu$ such that $\lambda_1=2n-1-m$, $\lambda_m=0$ and $\mu=(2n-m,\lambda_2,\dots,\lambda_m)$.
    \end{itemize}
    There is no arrow from $\HH_1$ to $\HH_2$.
  \item The Hasse diagram of $\IG(m,2n+1)$ is made of the union of :
   \begin{enumerate}
    \item the Hasse diagram $\HH_1$ of $\IG(m,2n)$, whose vertices are the cohomology classes associated to the Schubert varieties of $\IG(m,2n+1)$ not contained in the closed orbit ;
    \item the Hasse diagram $\HH_2$ of the closed orbit $\O\cong\IG(m-1,2n)$.
   \end{enumerate}
   The arrows from $\HH_1$ to $\HH_2$ are simple and of two types :
   \begin{itemize}
    \item $\lambda \ra \mu$ for $\lambda,\mu$ such that $\lambda_1=2n-m$ and $\mu=(2n+1-m,\lambda_2,\dots,\lambda_m)$ ;
    \item $\lambda \ra \mu$ for $\lambda,\mu$ such that $2(n-m)+\#\left\{1\leq i\leq m\mid\lambda_i\geq 1\right\}\leq\lambda_1\leq 2n-1-m$ and $\mu=(2n+1-m,\lambda_2,\dots,\lambda_{\lambda_1-2(n-m)},-1,\dots,-1)$.
   \end{itemize}
   There is no arrow from $\HH_2$ to $\HH_1$.
  \end{itemize}
\end{prop}

The proof is very similar to the $m=2$ case. However, the determination of the arrows between both parts of the Hasse diagram is a bit more complicated and uses a Pieri rule for the symplectic Grassmannian proved by Pragacz and Ratajski \cite[Thm 2.2]{PR}, hence we will not give it here.

\subsection{Embedding in the usual Grassmannian}
\label{sec:embed.A}

The easiest way to find a Giambelli formula for $\IGl$ is to use the Giambelli formula on $\mathrm{G}(2,2n+1)$ and to ``pull it back'' to $\IGl$.
More precisely, we use the natural embedding :
\[
\mathbf{j} : \IGl\hookrightarrow\mathrm{G}\left(2,2 n+1\right).
\]
This embedding identifies $\IGl$ with a generic hyperplane section of $\mathrm{G}(2,2n+1)$. So using the same arguments as for 
Lemma \ref{lem:restriction}, we can prove the
\begin{lem}
\label{lem:restriction.2}
\begin{itemize}
 \item If $a+b<2n-2$, then $\mathbf{j}^*\sigma_{a,b}=\tau_{a,b}$.
 \item If $a+b\geq 2n-2$, then
\[
 \mathbf{j}^*\sigma_{a,b}=\tau_{a,b}+\tau_{a+1,b-1}.
\]
\end{itemize}
\end{lem}
This proves that the map $\mathbf{j}^*$ is surjective and that its kernel is generated by the class 
\[
\sum_{i=0}^{n-1}(-1)^{n-i}\sigma_{n+i,n-i}.
\]

\subsection{Giambelli formula}
\label{sec:giambelli}

With Lemma \ref{lem:restriction.2} and the Giambelli formula for $\mathrm{G}(2,2n+1)$, we can prove a Giambelli formula with respect to $\tau_1$ 
and $\tau_{1,1}$. First define $d_r:=(\tau_{1^{1+j-i}})_{1\leq i,j\leq r}$, with the convention that $\tau_{1^p}=0$ if $p<0$ or $p>2$. We have the
\begin{prop}[Giambelli formula]
\begin{align*}
 \tau_{a,b}=\begin{cases}
        \tau_{1,1}^b d_{a-b} & \text{if $a+b\leq 2n-3$,} \\
        \sum_{q=0}^p (-1)^{p-q}\tau_{1,1}^{c-q}d_{2q}  & \text{if $(a,b)=(c+1+p,c-1-p)$,} \\
	\sum_{q=p}^{2n-2-c} (-1)^{q-p}\tau_{1,1}^{c-q}d_{2q+1}  & \text{if $(a,b)=(c+1+p,c-p)$,}
     \end{cases}
\end{align*}
where $n-1\leq c\leq 2n-2$ and $0\leq p\leq 2n-2-c$.
\end{prop}
We can also state a Giambelli formula expressing classes in terms of the $e_p:=c_p(\mathcal{Q})$ :
\begin{prop}[Another Giambelli formula]
 \begin{align*}
 \tau_{a,b}=\begin{cases}
        e_a e_b-e_{a+1} e_{b-1} & \text{if $a+b\leq 2n-3$,} \\
        (-1)^{a-n}e_{n-1}^2-e_a e_b+2\sum_{j=1}^{a-n}(-1)^{a-n-j}e_{n-1+j} e_{n-1-j} & \text{if $a+b=2n-2$,} \\
	e_a e_b+2\sum_{j=1}^{2n-1-a}(-1)^j e_{a+j} e_{b-j} & \text{if $a+b\geq 2n-1$}.
     \end{cases}
\end{align*}
\end{prop}

\subsection{Two presentations for the classical cohomology ring}
\label{sec:presentation}

\subsubsection{Presentation in terms of the classes $e_p$}

\begin{prop}[Presentation of $\H^*\left(\IGl,\Z\right)$]
\label{prop:pres.1}
The ring $\H^*\left(\IGl,\Z\right)$ is ge\-nerated by the classes $(e_p)_{1\leq p\leq 2n-1}$ and the relations are
\begin{align*}
\det\left(e_{1+j-i}\right)_{1\leq i,j\leq r} &= 0\text{ for }3\leq r\leq 2n, \tag{R1} \label{R1}\\
e_n^2+2\sum_{i\geq 1}e_{n+i}e_{n-i} &=0 \tag{R2} \label{R2}.
\end{align*}
\end{prop}

\begin{proof}
First of all, the quotient bundle $\mathcal{Q}$ of $\IGl$ is the pullback by the restriction map $\mathbf{i}$ of the quotient bundle 
$\mathcal{Q}^{+}$ on $\IG(2,2n+2)$. So the $\mathbf{i}^*c_p(\mathcal{Q}^{+})=c_p(\mathcal{Q})=e_p$ for $1\leq p\leq 2n$ generate 
$\H^*\left(\IGl,\Z\right)$. But $\mathcal{Q}$ having rank $2n-1$, $\mathbf{i}^*c_{2n}(\mathcal{Q}^{+})=0$, hence the cohomology ring of $\IGl$ is 
generated by the $(e_p)_{1\leq p\leq 2n-1}$.
Then we follow the method from \cite[Thm. 1.2]{BKT} to obtain presentations for the isotropic Grassmannians. Consider the graded ring
$A:=\Z\left[a_1,\dots,a_{2n-1}\right]$, where $\deg a_i=i$. Set $a_0=1$, and $a_i=0$ if $i<0$ or $i>2n-1$. We also define $d_0:=1$ and 
$d_r:=\det\left(a_{1+j-i}\right)_{1\leq i,j\leq r}$ for $r>0$. For all $r\geq 0$, set $b_r:=a_r^2+2\sum_{i\geq 1}(-1)^i a_{r+i}a_{r-i}.$
Now let $\phi : A\longrightarrow \H^*(\IGl,\Z)$ be the degree-preserving morphism of graded rings sending $a_i$ to $e_i$ for all $1\leq i\leq 2n-1$.
Since the $e_p$ generate $\H^*\left(\IGl,\Z\right)$, this morphism is surjective. 
To prove that relations $\eqref{R1}$ and $\eqref{R2}$ are satisfied, we must check that $\phi(d_r)=0$ for all $r>2$ and $\phi(b_n)=0$.
\begin{itemize}
 \item[\eqref{R1}]
Expanding the determinant $d_r$ with respect to the first column, we get the identity
\[
 d_r=\sum_{i=1}^r (-1)^{i-1}a_i d_{r-i}.
\]
Hence the identity on formal series :
\begin{equation}
 \left(\sum_{i=0}^{2n-1}a_i t^i\right)\left(\sum_{i\geq 0}(-1)^i d_i t^i\right)=1. \label{R3}
\end{equation}
On $\IGl$ we have the following short exact sequence of vector bundles
\begin{equation*}
 \SEC{\mathcal{S}}{\mathcal{O}_{\IGl}}{\mathcal{Q}},
\end{equation*}
so $c(\mathcal{S})c(\mathcal{Q})=1$, where $c$ denotes the total Chern class. But 
\[
 c(\mathcal{Q}) =\sum_{i=0}^{2n-2}\tau_i t^i,
\]
so \eqref{R3} implies
\[
 c(\mathcal{S})=\sum_{i\geq 0}(-1)^i\phi(d_i) t^i.
\]
Since $\mathcal{S}$ has rank $2$, it follows that $\phi(d_r)=0$ for all $r>2$, hence the relations \eqref{R1}.
 \item[\eqref{R2}] From the presentation of $\IG(2,2n+2)$ in \cite[Thm. 1.2]{BKT}, we know that
\[
 \upsilon_n^2+2\sum_{i\geq 1}(-1)^i\upsilon_{n+i}\upsilon_{n-i}=0
\]
in $\IG(2,2n+2)$. Pulling back by $\mathbf{i}$ we get \eqref{R2}.
\end{itemize}
Now consider the Poincar\'e polynomial of $\IGl$ computed in \cite[§ 2.2.3]{mihai2} :
\[
 P\left(\IG(m,2n+1),q\right)=\frac{\prod_{i=1}^l\left(q^{2n+2-2i}-1\right)\prod_{i=l+1}^{m}\left(q^{2n+4-2i}-1\right)}{(q^m-1)(q^{m-1}-1)\dots(q-1)}
\]
for $m=2l$.
Evaluating this polynomial at $q=1$, we get that the rank of $\H^*\left(\IG(2,2n+1)\right)$ is $2n^2$.

As in the proof of \cite[Thm. 1.2]{BKT}, we will need the following lemma :
\begin{lem}
\label{lem:free}
The quotient of the graded ring $\Z\left[a_1,\dots,a_{d}\right]$ with $\deg a_i=i$ modulo the relations
\[
 \det\left(a_{1+j-i}\right)_{1\leq i,j\leq r}=0,m+1\leq r\leq m+d
\]
is a free $\Z$-module of rank $\binom{m+d}{d}$.
\end{lem}
To prove the previous lemma notice that the above presentation is nothing but the presentation of the cohomology ring of the usual Grassmannian
$\mathrm{G}(m,m+d)$.
Now to conclude the proof of the proposition we use the
\begin{lem}
\label{lem:reid}
Let $A=\Z\left[a_1,\dots,a_{d}\right]$ be a free polynomial ring generated by homogeneous elements $a_i$ such that $\deg a_i=i$. 
Let $I$ be an ideal in $A$ generated by homogeneous elements $c_1,\dots,c_d$ in $A$ and $\phi :A/I\longrightarrow H$ 
be a surjective ring homomorphism. Assume :
\begin{enumerate}[C1.]
\item \label{cond1} $H$ is a free $\Z$-module of rank $\prod_i\left(\frac{\deg c_i}{\deg a_i}\right)$.
\item \label{cond2} For every field $K$, the $K$-vector space $\left(A/I\right)\otimes_{\Z}K$ has finite dimension.
\end{enumerate}
Then $\phi$ is an isomorphism.
\end{lem}
This result was proven in \cite[Lem. 1.1]{BKT}. Apply it for
\[
 H=\H^*(\IGl,\Z),\;I=\left(d_{3},\dots,d_{2n},b_n\right)\mbox{, and }A,\phi\mbox{ as above.}
\]
Condition \ref{cond1} is an immediate consequence of the rank calculation. For Condition \ref{cond2} it is enough to prove that $A/I$ is a quotient of
$A/\left(d_3,\dots,d_{2n+1}\right)$. Indeed, by Lemma \ref{lem:free}, the last module is a free $\Z$-module of finite rank. So we are left with 
proving that $d_{2n+1}$ belongs to the ideal $I$. But the following identities of formal series hold :
\begin{align*}
\left(\sum_{i=0}^{2n-1} a_i t^i\right)\left(\sum_{i=0}^{2n-1}(-1)^i a_i t^i\right) & =  \sum_{i=0}^{2n-1}(-1)^i b_i t^{2i}, \\
\left(\sum_{i=0}^{2n-1}(-1)^i a_i t^i\right)\left(\sum_{i\geq 0}d_i t^i\right) & =  1.
\end{align*}
Hence we get :
\[
 \sum_{i=0}^{2n-1} a_i t^i=\left(\sum_{i=0}^{2n-1}(-1)^i b_i t^{2i}\right)\left(\sum_{i\geq 0}d_i t^i\right).
\]
Modding out by the ideal $I$, it yields :
\[
 \sum_{i=0}^{2n-1} a_i t^i\equiv\left(\sum_{i=0}^{n-1}(-1)^i b_i t^{2i}+\sum_{i=n+1}^{2n-1}(-1)^i b_i t^{2i}\right)\left(\sum_{i=0}^2 d_i t^i+\sum_{i\geq 2n+1}d_i t^i\right).
\]
In degree $2n+1$, we get $0\equiv d_{2n+1}$, which ends the proof of the proposition.
\end{proof}

\subsubsection{Presentation in terms of $\tau_1$ and $\tau_{1,1}$}

First we will need a presentation for the symplectic Grassmannian $\IG(2,2n)$ in terms of $\upsilon_1$ and $\upsilon_{1,1}$ :
\begin{prop}
 \label{prop:pres.sympl}
The ring $\H^*\left(\IG(2,2n),\Z\right)$ is ge\-nerated by the classes $\upsilon_1$, $\upsilon_{1,1}$ and the relations are
\begin{align*}
\frac{1}{\upsilon_1}\det\left(\upsilon_{1^{1+j-i}}\right)_{1\leq i,j\leq 2n-1} &= 0,\\
\det\left(\upsilon_{1^{1+j-i}}\right)_{1\leq i,j\leq 2n} &= 0
\end{align*}
\end{prop}

\begin{proof}
We will use Lemma \ref{lem:reid}.
 Set $R:=\Z\left[a_1,a_2\right]$, where $\deg a_i=i$. We denote by $\phi : R\ra \H^*\left(\IG(2,2n),\Z\right)$ the surjective ring 
homomorphism given by $a_i\mapsto\tau_{1^i}$. We also use the convention that $a_0=1$ and $a_i=0$ for $i\not\in\left\{0,1,2\right\}$. 
For $r\geq 1$, set $\delta_r:=\det\left(a_{1+j-i}\right)_{1\leq i,j\leq r}$. 
We have the recurrence relation
\begin{equation}
\label{eq:rec.delta}
 \delta_r=a_1\delta_{r-1}-a_2\delta_{r-2},
\end{equation}
which is equivalent to the identity of formal series
\[
 \left(\sum a_i t^i\right) \left(\sum (-1)^i\delta_i t^i \right) =1.
\]
But $\phi(a_i)=\tau_{1^i}=c_i(\mathrm{S}^*)$. Moreover, as
\[
 \SEC{\mathrm{S}^\perp}{\mathcal{O}_{\IG}}{\mathrm{S}^*},
\]
where we denote by $\mathrm{S}$ the tautological bundle on $\IG(2,2n)$,
we have $c(\mathrm{S}^\perp)c(\mathrm{S}^*)=1$, hence $\delta_r=c_r\left( (\mathrm{S}^\perp)^*\right)=c_r(\mathrm{Q})$ ($\mathrm{Q}$ being the 
quotient bundle on $\IG(2,2n)$). Since $\mathrm{Q}$ has rank $2n-2$, we have $\phi(\delta_r)=0$ for all $r>2n-2$, and in particular we get 
$\phi(\delta_{2n-1})=\phi(\delta_{2n})=0$. 
We can write $\delta_{2q+1}$ as 
\[
 \delta_{2q+1}=a_1 P_q(a_1,a_2),
\]
where $P_q(a_1,a_2)$ is a homogeneous polynomial of degree $2q$. Now set $\delta_{2q+1}':=P_q(a_1,a_2)$. We want to prove that 
$\phi(\delta_{2n-1}')=0$. For this, since $\IG(2,2n)$ is a hyperplane section of the usual Grassmannian $\mathrm{G}(2,2n)$, we use Lefschetz's 
theorem. In particular, we obtain that the multiplication by the hyperplane class $\upsilon_1$ is surjective from $\H^{2n-2}(\IG(2,2n),\Z)$ to 
$\H^{2n-1}(\IG(2,2n),\Z)$. But these vector spaces have the same dimension $n-1$, so it is bijective. As we already know that $\phi(\delta_{2n-1})=0$ it 
implies that $\phi(\delta_{2n-1}')=0$.
Now let $I:=(\delta_{2n-1}',\delta_{2n})$. We proved that $\phi(I)=0$ so we may define 
$\overline{\phi} : R\slash I\ra\H^*\left(\IG(2,2n),\Z\right)$. Now check that Conditions $\ref{cond1}$ and $\ref{cond2}$ are satisfied :
\begin{enumerate}
 \item[(C1)] $\H^*\left(\IG(2,2n),\Z\right)$ is a free $\Z$-module of rank $2n(n-1)=\frac{\deg(d_{2n-1})'\deg(d_{2n})}{\deg a_1\deg a_2}$.
 \item[(C2)] For every field $K$, $\left( R\slash I\right)\otimes_{\Z} K$ is finite-dimensional. Indeed $R\slash I$ is a quotient of 
$R\slash (d_{2n-1},d_{2n})$, which is isomorphic with $\H^*\left(\mathrm{G}(2,2n),\Z\right)$, hence a free $\Z$-module of finite rank.
\end{enumerate}
Finally Lemma \ref{lem:reid} yields that $\overline{\phi}$ is an isomorphism, hence the result.
\end{proof}

Now we deduce a presentation of $\H^*\left(\IGl,\Z\right)$ using classes $\tau_1$ and $\tau_{1,1}$ :
\begin{prop}[another presentation of $\H^*\left(\IGl,\Z\right)$]
\label{prop:pres.2}
The ring $\H^*\left(\IGl,\Z\right)$ is ge\-nerated by the classes $\tau_1$, $\tau_{1,1}$ and the relations are
\begin{align*}
\det\left(\tau_{1^{1+j-i}}\right)_{1\leq i,j\leq 2n} &= 0,\\
\frac{1}{\tau_1}\det\left(\tau_{1^{1+j-i}}\right)_{1\leq i,j\leq 2n+1} &= 0
\end{align*}
\end{prop}

\begin{proof}
 First notice that $\tau_1$ and $\tau_{1,1}$ generate the cohomology ring of $\IGl$ since they are the pullbacks of the Chern classes of the dual 
tautological bundle over $\mathrm{G}(2,2n+1)$ by the surjective restriction map $\mathbf{j}$. Then define $R:=\Z\left[a_1,a_2\right]$, where 
$\deg a_i=i$. We denote by $\phi : R\ra \H^*\left(\IGl,\Z\right)$ the surjective ring homomorphism given by $a_i\mapsto\tau_{1^i}$. We also use 
the convention that $a_0=1$ and $a_i=0$ for $i\not\in\left\{0,1,2\right\}$. 
For $r\geq 1$, set $\delta_r:=\det\left(a_{1+j-i}\right)_{1\leq i,j\leq r}$. 
On $\mathrm{G}(2,2n+1)$ we know by the usual presentation (see for instance \cite[§ 3]{ST}) that
\[
 \det\left(\sigma_{1^{1+j-i}} \right)_{1\leq i,j\leq 2n}=0
\]
Now define $\delta_{2q+1}'$ as in the proof of Proposition \ref{prop:pres.sympl}. 
Using the embedding in the symplectic Grassmannian $\IG(2,2n+2)$, we get that $\phi(\delta_{2n+1}')=0$. Indeed, we only have to pull back 
the relation $\frac{1}{\upsilon_1}\det\left(\upsilon_{1^{1+j-i}}\right)_{1\leq i,j\leq 2n+1} = 0$ proven in Proposition \ref{prop:pres.sympl}.
Finally, set $I=(d_{2n},d_{2n+1}')$ and apply Lemma \ref{lem:reid}.
\end{proof}

% Cohomologie quantique

\section{Quantum cohomology}

Our main goal in this section is to prove a quantum Pieri formula for $\IGl$. We denote the quantum product of two classes $\tau_\lambda$ and 
$\tau_\mu$ as $\tau_\lambda\star\tau_\mu$. The degree of the quantum parameter $q$ is equal to the index of $\IGl$, so $\deg q=2n$.
\begin{theo}[Quantum Pieri rule for $\IGl$]
\label{theo:qpieri}
\begin{align*}
 \tau_1\star\tau_{a,b}&=\begin{cases}
	\tau_{a+1,b}+\tau_{a,b+1} & \text{if $a+b\neq 2n-3$ and $a\neq 2n-1$,} \\
	\tau_{a,b+1}+2\tau_{a+1,b}+\tau_{a+2,b-1} & \text{if $a+b=2n-3$,} \\
	\tau_{2n-1,b+1}+q\tau_b & \text{if $a=2n-1$ and $0\leq b\leq 2n-3$,} \\
	q(\tau_{2n-1,-1}+\tau_{2n-2}) & \text{if $a=2n-1$ and $b=2n-2$.}
     \end{cases} \\
 \tau_{1,1}\star\tau_{a,b}&=\begin{cases}
	\tau_{a+1,b+1} & \text{if $a+b\neq 2n-4$, $2n-3$ and $a\neq 2n-1$,} \\
	\tau_{a+1,b+1}+\tau_{a+2,b} & \text{if $a+b=2n-4$ or $2n-3$,} \\
	q\tau_{b+1} & \text{if $a=2n-1$ and $b\neq 2n-3$,} \\
	q(\tau_{2n-1,-1}+\tau_{2n-2}) & \text{if $a=2n-1$ and $b=2n-3$.}
     \end{cases}\end{align*}
\end{theo}
The previous theorem is proved in \ref{sec:qpieri}, and from this a quantum presentation is deduced in \ref{sec:qpres}.
To prove the quantum Pieri formula, we first study in \ref{sec:mod} the moduli spaces of stable maps of degree $1$ to $\IGl$. Then in \ref{sec:trans} we decribe conditions for the Gromov-Witten invariants to have enumerative meaning. Finally, in 
\ref{sec:mod.2} and \ref{sec:mod.3} we compute the invariants we need. From now on, we denote $\IGl$ by $\IG$.

\subsection{The moduli spaces $\Mod{r}{1}$}
\label{sec:mod}

If $X$ is a smooth projective variety we denote by $\overline{\mathcal{M}}_{g,n}\left(X,\beta\right)$ the moduli space of stable $n$-pointed maps 
$f$ in genus $g$ to $X$ with degree $\beta\in\H_2(X,\Z)$. This moduli space is endowed with $n$ evaluation maps $(ev_i)_{1\leq i\leq n}$ mapping a stable map $f$ to its value at the $i^\text{th}$ marked point. We refer to \cite{FP} for more details. If $X$ has Picard rank $1$, which is the case when $X=\IGl$, then $\beta=d H$ for some $d\geq 0$, $H$ being the positive generator of the Picard group. In this situation, we will simply denote the degree as the integer $d$. In this section we prove the
\begin{prop}
\label{prop:mod}
For every $r \in \N$, the moduli space $\Mod{r}{1}$ is a smooth projective variety. Moreover, it has the expected dimension $6n-6+r$.
\end{prop}

\begin{proof}
To prove this, we use a remark of Fulton and Pandharipande in \cite[§ 0.4]{FP}, which states that for all $r \geq 1$, the moduli space $\M_{0,r}(\P{m},1)$ is a locally trivial fibration over the variety $\G(\P{1},\P{m})$ of lines in $\P{m}$, having $\M_{0,r}(\P{1},1)$ as a fiber. Moreover, this last moduli space is isomorphic to the configuration space $\P{1}[r]$ of Fulton-MacPherson. The fibration is simply the map
\[
\begin{array}{ccc}
 \M_{0,r}(\P{m},1) & \ra & \G(\P{1},\P{m}) \\
 \left[ f : (C;p_1,\dots,p_r) \ra \P{m} \right] & \mapsto & f(C)
\end{array}
\]
which to any stable map of degree one associates its image line.

The Plücker embedding embeds $\IG$ as a closed subvariety of a projective space $\P{m}$ (with $m=(2n-1)(n+1)$). Under this embedding, lines in $\IG$ are lines in $\P{m}$. From \cite[§ 5.1]{FP}, we know that this yields an embedding $\Mod{r}{1} \hookrightarrow \M_{0,r}(\P{m},1)$. If we denote by $Y_1$ the variety of lines on $\IG$, we get a commutative diagram :

\begin{center}
 \begin{tikzpicture}
  \matrix [matrix of math nodes,row sep=1cm, column sep=1cm,text height=1.5ex,text depth=0.25ex]
  {
	 |(M1)| \overline{\mathcal{M}}_{0,r}(\mathbb{P}^m,1) & |(G)| \mathrm{G}(\mathbb{P}^1,\mathbb{P}^m) \\
	 |(M2)| \overline{\mathcal{M}}_{0,r}(\mathrm{IG},1) & |(Y1)| Y_1 \\
};
\path[->]
(M1) edge node[auto] {} (G)
(M2) edge node[auto] {} (Y1)
;
\path[right hook->]
(M2) edge node[auto] {} (M1)
(Y1) edge node[auto] {} (G)
;
\end{tikzpicture}
\end{center}

Hence the map $\Mod{r}{1} \ra Y_1$ is also locally trivial. Since the fiber $\P{1}[r]$ is known to be smooth, we only need to prove that the variety of lines $Y_1$ is also smooth. 

First notice that lines in $\IG$ are of the form
\[
 \mathcal{D}(U_1,U_3) := \left\{ \Sigma \in \IGl \mid U_1 \subset \Sigma \subset U_3 \right\}
\]
where $\dim U_i = i$ and $U_3\subset U_1^\perp$. Hence $Y_1$ is a subvariety of the (type $A$) flag variety $\mathrm{F}(1,3 ; 2n+1)$. Let us denote by $\SS_1$ and $\SS_3$ the tautological bundle on $\mathrm{F}(1,3;2n+1)$ and consider the homogeneous vector bundle $\mathcal{E}:=\SS_1^*\otimes\left(\SS_3/\SS_1\right)^*$ on $\mathrm{F}(1,3;2n+1)$. Let also $\pi$ be the projection map from the complete flag variety $\mathrm{F}(\C^{2n+1})$ to the two-step flag variety $\mathrm{F}(1,3;2n+1)$. If we denote by $\UU_1,\dots,\UU_{2n+1}$ the tautological bundles on $\mathrm{F}(\C^{2n+1})$, we see that $\mathcal{E}=\pi_*\left(\UU_1^{-1}\otimes\left(\UU_2/\UU_1\right)^{-1}\right)$. Hence
\[
 \H^0\left(\mathrm{F}(1,3;2n+1),\mathcal{E}\right) = \H^0\left(\mathrm{F}(\C^{2n+1}),\UU_1^{-1}\otimes\left(\UU_2/\UU_1\right)^{-1}\right) = {\bigwedge}^2 \left(\C^{2n+1}\right)^*,
\]
the last equality being a consequence of the Borel-Weil theorem. This implies that the form $\omega$ is a generic section of the vector bundle $\SS_1^*\otimes\left(\SS_3/\SS_1\right)^*$ on $\mathrm{F}(1,3;2n+1)$. From the condition $U_3\subset U_1^\perp$, it follows that the zero locus of this section is exactly the variety $Y_1$. Moreover, the vector bundle $\mathcal{E}$ is generated by its global sections. Hence $Y_1$ is smooth, and so is $\Mod{r}{1}$. 

Finally, the dimension of $Y_1$ is equal to the dimension of its open orbit under the action of the odd symplectic group $\Sp$ :
\[
 \OO := \left\{ (U_1,U_3) \mid U_3\not\supset K, U_3 \text{ not isotropic} \right\}.
\]
Since $\dim \OO = \dim \G(1,2n+1) + \dim \G(2,U_1^\perp/U_1) = 6n-6$, it follows that $\dim Y_1 = 6n-6$, and finally $\dim \Mod{r}{1} = \dim Y_1 + \dim \P{1}[r] = 6n-6+r$.
\end{proof}

\subsection{Enumerativity of the invariants in $\Mod{2}{1}$ and $\Mod{3}{1}$}
\label{sec:trans}

In this section we will use a Kleiman-type lemma for quasi-homogeneous spaces, due to Graber in \cite[Lem. 2.5]{graber} :
\begin{lem}
\label{lem:graber}
Let $X$ be a variety endowed with an action of a connected algebraic group $G$ with only finitely many orbits and $Z$ an irreducible scheme with a 
morphism $f : Z\ra X$. Let $Y$ be a subvariety of $X$ that intersects the orbit stratification properly. Then there exists a dense open subset 
$U$ of $G$ such that $\forall g\in U$, $f^{-1}(g Y)$ is either empty or has pure dimension $\dim Y+\dim Z-\dim X$. Moreover, if $X$, $Y$ and $Z$ 
are smooth and we denote by $Y_{\mathrm{reg}}$ the subset of $Y$ along which the intersection with the stratification is transverse, then the 
(possibly empty) open subset $f^{-1}(g Y_{\mathrm{reg}})$ is smooth.
\end{lem}

This enables us to prove the following enumerativity result for degree one Gromov-Witten invariants on $\IGl$. 
\begin{theo}[Enumerativity of the Gromov-Witten invariants]
\label{theo:enum}
Let $r$ be a positive integer and $Y_1,\dots,Y_r$ be subvarieties of $\IG$ of codimension at least $2$ intersecting the closed orbit generically transversely and representing cohomology classes $\gamma_1,\dots,\gamma_r$  such that 
\[
 \sum_{i=1}^r\codim\gamma_i=\dim\Mod{r}{1}.
\]
Then there exists a dense open subset $U\subset \Sp^r$ such that for all $g_1,\dots,g_r\in U$, the Gromov-Witten invariant $I_1(\gamma_1,\dots,\gamma_r)$ is equal to the number of lines of $\IG$ incident to the translates $g_1 Y_1,\dots,g_r Y_r$.
\end{theo}

\begin{proof}
The result is proven by successively applying the Transversality lemma \ref{lem:graber}.
 First we prove that stable maps with reducible source do not contribute to the Gromov-Witten invariant by applying the lemma to the following 
diagram 
\begin{center}
\begin{tikzpicture}
  \matrix [matrix of math nodes,row sep=1cm, column sep=1cm,text height=1.5ex,text depth=0.25ex]
  {
	  & |(M)| \overline{\mathcal{M}}\setminus \mathcal{M}^{*}\\
	 |(Y)| \underline{Y} & |(IG)| \mathrm{IG}^r \\
};
\path[->]
(M) edge node[auto] {$\underline{ev}$} (IG)
;
\path[right hook->]
(Y) edge node[auto] {$i$} (IG)
;
\end{tikzpicture}
\end{center}
where $\underline{Y}=(Y_1,\dots,Y_r)$, $\underline{ev}=ev_1\times\dots\times ev_r$, $\overline{\mathcal{M}}=\Mod{r}{1}$ and $\mathcal{M}^{*}$ is 
the locus of map with irreducible source, which is a dense open subset by Proposition \ref{prop:mod}.

We should also prove that it is not possible for a line to be incident to one of the subvarieties $Y_i$ in more than one point, since such a line 
would contribute several times to the invariant. Suppose for example that there exists a line $L$ that intersects $Y_1$ in at least two points. Then any stable map $f$ whose image curve is $L$ corresponds to a map $\tilde{f}$ in $\Mod{r+1}{1}$ (in fact in $\mathcal{M}_{0,r+1}^*\left(\IG,1\right)$) that contributes to the invariant $I_1(\gamma_1,\gamma_1,\dots,\gamma_r)$. By Proposition \ref{prop:mod}, $\Mod{r+1}{1}$ has dimension $6n-5+r$. Hence applying Lemma \ref{lem:graber} to the following diagram
\begin{center}
\begin{tikzpicture}
  \matrix [matrix of math nodes,row sep=1cm, column sep=1cm,text height=1.5ex,text depth=0.25ex]
  {
	  & |(M)| \mathcal{M}_{0,r+1}^*\left(\mathrm{IG},1\right)\\
	 |(Y)| Y_1\times Y_1\times\dots\times Y_r & |(IG)| \mathrm{IG}^{r+1} \\
};
\path[->]
(M) edge node[auto] {$\underline{ev}$} (IG)
;
\path[right hook->]
(Y) edge node[auto] {$i$} (IG)
;
\end{tikzpicture}
\end{center}
and using the fact that $\codim\gamma_1\geq 2$ we conclude that such a line cannot exist.

Now using 
\begin{center}
\begin{tikzpicture}
  \matrix [matrix of math nodes,row sep=1cm, column sep=1cm,text height=1.5ex,text depth=0.25ex]
  {
	  & |(M)| \mathcal{M}^{*}\\
	 |(Y)| \mathrm{Sing}\;\underline{Y} & |(IG)| \mathrm{IG}^r \\
};
\path[->]
(M) edge node[auto] {$\underline{ev}$} (IG)
;
\path[right hook->]
(Y) edge node[auto] {$i$} (IG)
;
\end{tikzpicture}
\end{center}
where $\mathrm{Sing}\;\underline{Y}$ denotes the singular locus of $\underline{Y}$, we may assume that $\underline{Y}$ is smooth.
Moreover, since $Y_1,\dots,Y_r$ intersect the closed orbit generically transversely, a third application of Lemma \ref{lem:graber} allows us to 
assume that this intersection is transverse everywhere. 
Finally, applying the lemma to 
\begin{center}
\begin{tikzpicture}
  \matrix [matrix of math nodes,row sep=1cm, column sep=1cm,text height=1.5ex,text depth=0.25ex]
  {
	  & |(M)| \mathcal{M}^{*}\\
	 |(Y)| \underline{Y} & |(IG)| \mathrm{IG}^r \\
};
\path[->]
(M) edge node[auto] {$\underline{ev}$} (IG)
;
\path[right hook->]
(Y) edge node[auto] {$i$} (IG)
;
\end{tikzpicture}
\end{center}
we conclude that there exists a dense open subset $U\subset\Sp^r$ such that for all $(g_1,\dots,g_r)\in U$, $\bigcap_{i=1}^r ev_i^{-1}(g_i Y_i)$ is
a finite number of reduced points, which equals the number of lines incident to all the $g_i Y_i$.
\end{proof}

\begin{rem}
 Theorem \ref{theo:enum} enables us to compute the Gromov-Witten invariants by geometric means. However, Schubert varieties will not be appropriate to perform this calculation. Indeed, the intersection of any Schubert variety and the closed orbit is not even proper. So we will instead use the restrictions of the Schubert varieties of the usual Grassmannian.
\end{rem}

\subsection{Computation of the invariants in $\Mod{2}{1}$}
\label{sec:mod.2}

In this paragraph, we use Theorem \ref{theo:enum} to compute all invariants of $\IG$ of the form $I_1(\alpha,\beta)$, where $\alpha$ and $\beta$ are the classes of the restriction to $\IG$ of some Schubert varieties $Y_1$ and $Y_2$ of the usual Grassmannian, defined with respect to complete flags $F_\bullet$ and $G_\bullet$. In order for the varieties $Y_1$ and $Y_2$ to verify the conditions of the theorem, we will need some technical conditions to hold for the defining flags $F_\bullet$ and $G_\bullet$. We state these conditions in Lemma \ref{lem:gen.2} and prove that they hold for generic flags : this is quite straightforward, and the list of conditions is in fact longer to state than to check. Then we compute the invariants in Proposition \ref{prop:mod.2}.

\begin{nota}
 Denote by
\begin{itemize}
 \item $\mathbb{F}_n$ the variety of complete flags in $\C^{2n+1}$ ;
 \item $\Lambda_n$ the variety of antisymmetric $2$-forms with maximal rank on $\C^{2n+1}$.
\end{itemize}
\end{nota}

\begin{lem}
\label{lem:gen.2}
 Assume $n\geq 2$.
 Then the set of triples $(F_{\bullet},G_{\bullet},\omega)\in\mathbb{F}_n\times\mathbb{F}_n\times\Lambda_n$ such that the following holds
\begin{description}
 \item[(C1)] $\forall 0\leq p\leq 2n+1$, $\omega_{\mid F_p}$ has maximal rank ;
 \item[(C2)] $\forall 0\leq p\leq 2n+1$, $\omega_{\mid G_p}$ has maximal rank ;
 \item[(C3)] $\forall 0\leq p,q\leq 2n+1$, $F_p\cap G_q$ has the expected dimension ;
 \item[(C4)$_i$] $\dim \left( F_{2n+1-i}\cap G_{i+3}\cap F_1^{\perp}\cap G_1^{\perp}\right)=1$ ; $\mathbf{(0\leq i\leq 2n-2)}$ ;
 \item[(C5)$_i$] $\dim F_{2n-i}\cap G_{i+3}\cap G_1^{\perp}=1$ and $\dim (F_{2n-i}\cap G_{i+3}\cap G_1^{\perp})^{\perp}\cap F_2=1$; $\mathbf{(0\leq i\leq 2n-2)}$ ;
 \item[(C6)$_i$] $\dim F_{2n+1-i}\cap G_{i+2}\cap F_1^{\perp}=1$ and $\dim (F_{2n+1-i}\cap G_{i+2}\cap F_1^{\perp})^{\perp}\cap G_2=1$; $\mathbf{(2\leq i\leq 2n-4)}$ ;
 \item[(C7)$_i$] $\dim \left(F_{2n-i}\cap G_{i+2}\right)^{\perp}\cap F_2=1$ ; $\mathbf{(2\leq i\leq 2n-4)}$ ;
 \item[(C8)$_i$] $\dim\left( F_{2n-i}\cap G_{i+2}\right)^{\perp}\cap G_2=1$ ; $\mathbf{(2\leq i\leq 2n-4)}$ ;
 \item[(C9)] $F_1\not\subset G_1^{\perp}$ ;
 \item[(C10)] $G_1\not\subset F_1^{\perp}$ ;
 \item[(C11)$_i$] $F_{2n-1-i}\cap G_{i+3}\cap G_1^{\perp}=0$ ; $\mathbf{(0\leq i\leq 2n-6)}$ ;
 \item[(C12)$_i$] $F_{2n+1-i}\cap G_{i+1}\cap F_1^{\perp}=0$ ; $\mathbf{(4\leq i\leq 2n-2)}$ ;
\end{description}
is a dense open subset in $\mathbb{F}_n\times\mathbb{F}_n\times\Lambda_n$.
\end{lem}

\begin{proof}
 $\mathbb{F}_n\times\mathbb{F}_n\times\Lambda_n$ is a (quasi-projective) irreducible variety. Moreover all conditions are clearly open. So it is 
enough to show that each of them is non-empty.
\begin{description}
 \item[(C1),(C2) et (C3)] Obvious.
 \item[(C4)$_i$] Since $n\geq 2$, we may choose the flags $F_{\bullet}$ and $G_{\bullet}$ such that the subspace $A:=F_{2n+1-i}\cap G_{i+3}$ has 
dimension $3$ and $A$ together with the lines $L:=F_1$ and $L':=G_1$ are in direct sum. Then there exists a form $\omega\in\Lambda_n$ such that 
$A\cap L^\perp \cap L'^\perp$ has dimension $1$.
 \item[(C5)$_i$] As before we may choose $F_{\bullet}$ and $G_{\bullet}$ such that $A:=F_{2n-i}\cap G_{i+3}$ has dimension $2$ and $A$, $L:=G_1$ 
and $B:=F_2$ are complementary. So we may construct $\omega\in\Lambda_n$ such that $(A\cap L^\perp)^\perp\cap B$ has dimension $1$. First 
construct $\omega_0$ on $A\oplus B\oplus L$. Let $a\in A\setminus 0$ and $b\in B\setminus 0$. There exists $\omega_0$ a symplectic form on 
$A\oplus B$ such that $\omega_0(a,b)\neq 0$. Then we extend $\omega_0$ to $\omega$ defined on $A\oplus B\oplus L$ by setting $\omega(a,l)=0$, 
$\omega(a',l)\neq 0$ and for instance $\omega(\beta,l)=0$ for all $\beta\in B$, where $l$ generates $L$ and $a,a'$ generate $A$.
 \item[(C6)$_i$] As in \textbf{(C5)}$_i$.
 \item[(C7)$_i$] We may choose $F_{\bullet}$ and $G_{\bullet}$ such that $L:=F_{2n-i}\cap G_{i+2}$ has dimension $1$ and is in direct sum with 
$A:=F_2$. But then there exists $\omega\in\Lambda_n$ such that $A\not\subset L^\perp$.
 \item[(C8)$_i$] As in \textbf{(C7)}$_i$.
 \item[(C9)] $G_1^{\perp}$ is a general hyperplane, so it does not contain $F_1$. 
 \item[(C10)] As in \textbf{(C9)}.
 \item[(C11)$_i$] $F_{2n-1-i}\cap G_{i+3}$ is a line and $G_1^{\perp}$ is a general hyperplane, so their intersection is zero.
 \item[(C12)$_i$] As in \textbf{(C11)$_i$}. \qedhere
\end{description}
\end{proof}

We can now define the varieties we will use to compute the invariants, which will be restrictions of the Schubert varieties of the usual 
Grassmannian :

\begin{lem}
\label{lem:cycle}
 Let $0\leq j\leq n-1$ and $0\leq i\leq 2n-1-2j$ be integers. Let 
\begin{displaymath}
 X_{i,j}:=\left\{\Sigma\in G\mid \Sigma\cap F_{j+1}\neq 0,\Sigma\subset F_{2n+1-i-j}\right\},
\end{displaymath}
be a subvariety of $G:=G(2,2n+1)$, where $F_{\bullet}$ is a complete flag satisfying condition \emph{\textbf{(C1)}}. Then~:
\begin{enumerate}
 \item $X_{i,j}$ and $\IG$ intersect generically transversely. 
 \item Let $Y_{i,j}:=X_{i,j}\cap\IG$. We have
\begin{align*}
   [Y_{i,j}]^{\IG}&=
     \begin{cases}
        \tau_{2n-1-j,i+j}+\tau_{2n-j,i+j-1} & \text{if $j\neq 0$ and $i\neq 2n-1-2j$,} \\
        \tau_{2n-j,2n-2-j}  & \text{if $j\neq 0$ and $i=2n-1-2j$,} \\
	\tau_{2n-1,i} &\text{if $j=0$ and $i\neq 2n-1$,} \\
	0 &\text{if $j=0$ and $i=2n-1$},
     \end{cases}
   \end{align*}
where we denote by $[V]^\IG$ (respectively by $[V]^G$) the class of the subvariety $V$ in $\IG$ (respectively in $G$).
\end{enumerate}
\end{lem}

\begin{proof}
\begin{enumerate}
 \item In the Schubert cell $C_{i,j}\subset X_{i,j}$, a direct computation shows that $\T{p}{X_{i,j}}\not\subset\T{p}{\IG}$ as soon as 
$F_{j+1}\not\subset F_{2n+1-i-j}^{\perp}$, which is true by condition \textbf{(C1)}. So $C_{i,j}\cap\IG$ is transverse. Applying again 
\textbf{(C1)}, we notice that $C_{i,j}\cap\IG$ is an open subset of $X_{i,j}\cap\IG$.
 \item We have $[X_{i,j}]^G=\sigma_{2n-1-j,i+j}$. Moreover, the previous item implies that $[Y_{i,j}]^G=\sigma_1 [X_{i,j}]^G$.
So
\begin{align*}
 [Y_{i,j}]^{G}&=
     \begin{cases}
        \sigma_{2n-1-j,i+j+1}+\sigma_{2n-j,i+j} & \text{if $j\neq 0$ and $i\neq 2n-1-2j$,} \\
        \sigma_{2n-j,2n-1-j}  & \text{if $j\neq 0$ and $i=2n-1-2j$,} \\
	\sigma_{2n-1,i+1} &\text{if $j=0$ and $i\neq 2n-1$,} \\
	0 &\text{if $j=0$ and $i=2n-1$}.
     \end{cases}
   \end{align*}
Moreover, $[Y_{i,j}]^G=\mathbf{j}_{\star}[Y_{i,j}]^{\IG}$, $[Y_{i,j}]^{\IG}=\sum_{p=0}^{\lfloor n-1-\frac{i}{2}\rfloor}\alpha_p\tau_{2n-1-p,i+p}$ 
and $\mathbf{j}_{\star}\tau_{a,b}=\sigma_{a,b+1}$ for $a+b\geq 2n-1$, so we can determine the $\alpha_p$ by identifying both expressions. \qedhere
\end{enumerate}
\end{proof}

We now assume all genericity conditions \textbf{(C1-12)} are satisfied and prove the

\begin{prop}
\label{prop:mod.2}
 Let $0\leq i\leq 2n-2$, $0\leq 2j\leq 2n-2-i$ and $0\leq 2l\leq i$ be integers. Set $Y_1:=Y_{i,j}(F_{\bullet})$ and 
$Y_2:=Y_{2n-2-i,l}(G_{\bullet})$, where the complete flags $F_{\bullet}$ and $G_{\bullet}$ as well as the form $\omega$ verify the transversality
 conditions of Lemma \ref{lem:gen.2}. Then
\begin{enumerate}
 \item The intersections $Y_1\cap\O$ et $Y_2\cap\O$ are transverse. Moreover
  \begin{align*}
   Y_1\cap\O&=
     \begin{cases}
        \emptyset & \text{if $i$ or $j\neq 0$} \\
        \left\{F_1\oplus K\right\}  & \text{if $i=j=0$}
     \end{cases}  
 \\
Y_2\cap\O&=
       \begin{cases}
          \emptyset & \text{if $i\neq 2n-2$ or $l\neq 0$} \\
          \left\{G_1\oplus K\right\}  & \text{if $i=2n-2$ and $l=0$}
       \end{cases}
  \end{align*}
 \item If $j$ or $l\geq 2$, there exists no line passing through $Y_1$ and $Y_2$. Else there exists a unique line passing through $Y_1$ and $Y_2$. Therefore the Gromov-Witten invariant $I_1(\tau_{a,b},\tau_{c,d})$ equals $1$ if $a=c=2n-1$ and $b+c=2n-2$, and $0$ in the other cases.
\end{enumerate}
\end{prop}

\begin{proof}
 \begin{enumerate}
 \item $Y_1\cap\O=\left\{\Sigma\in\IG\mid \Sigma\cap F_{j+1}\neq 0,K\subset\Sigma\subset F_{2n+1-i-j}\right\}$, so if $i+j\neq 0$, then 
$K\subset F_{2n+1-i-j}$, which, according to \textbf{(C1)}, implies that $Y_1\cap\O=\emptyset$, so the intersection is transverse. Moreover if 
$i+j=0$ we get $Y_1\cap\O=\left\{F_1\oplus K\right\}$. Denote by $\Sigma_0$ the point $K\oplus F_1$. To prove transversality at $\Sigma_0$ we use
the embedding in the usual Grassmannian $G:=G(2,2n+1)$. It is well-known that $\T{\Sigma_0}{G}=\Hom{\Sigma_0}{\C^{2n+1}\slash\Sigma_0}$. Now 
express $\T{\Sigma_0}{Y_1}$ and $\T{\Sigma_0}{\O}$ as subspaces of $\T{\Sigma_0}{G}$ :
\begin{align*}
 \T{\Sigma_0}{Y_1}=\left\{\phi\in\T{\Sigma_0}{G}\mid\phi(f_1)=0\right\}, \\
 \T{\Sigma_0}{\O}=\left\{\phi\in\T{\Sigma_0}{G}\mid\phi(k)=0\right\},
\end{align*}
where $f_1$ and $k$ generate $F_1$ and $K$. We see that these subspaces are complementary in $\T{\Sigma_0}{G}$. Computing $\dim Y_1=2n-2$ and 
 $\dim\O=2n-1$ we conclude that they generate $\T{\Sigma_0}{\IG}$.
We can proceed in a similar fashion for $Y_2\cap \O$.
 \item We first study the case where $j$ or $l\geq 2$. Let $\mathcal{D}:=\mathcal{D}(V,W)$ be a line meeting $Y_1$ and $Y_2$. Then we must have $V\subset F_{2n+1-i-j}\cap G_{i+3-l}$. But 
according to \textbf{(C3)}, this subspace is either zero or it has codimension $2n+4-j-l$. So for $j+l\geq 3$, it is zero and there is no line. If 
$j=2$ and $l=0$ (and symmetrically if $j=0$ and $l=2$), we must have $V\subset F_{2n-1-i}\cap G_{i+3}\cap G_1^{\perp}=0$, which is impossible by 
\textbf{(C11)$_i$} (respectively by \textbf{(C12)$_i$}). So for a line to exist we must have $j$ and $l\leq 1$.

 Now assume $j,l\leq 1$. There are four cases to study :
 \begin{enumerate}[a)]
 \item $j=l=0$ ;
 \item $j=1$, $l=0$ ;
 \item $j=0$, $l=1$ ;
 \item $j=l=1$.
\end{enumerate}
\begin{enumerate}[a)]
 \item Let $A=F_{2n+1-i}\cap G_{i+3}$. We have $\dim A=3$ by \textbf{(C3)}. But $V\subset A$ and $V\subset F_1^{\perp}\cap G_1^{\perp}$ since 
$F_1,G_1\subset W$ and $W\subset V^{\perp}$. By \textbf{(C4)}$_i$, we have $\dim A\cap F_1^{\perp}\cap G_1^{\perp}=1$, hence 
$V=A\cap F_1^{\perp}\cap G_1^{\perp}$. So $W\supset V +(F_1\oplus G_1)$ ($F_1$ and $G_1$ are in direct sum by \textbf{(C3)}). To show equality, it is 
enough to prove that the sum is direct. If not then there exists a non-zero vector of the form $a f_1+ b g_1$ in $V$, where $f_1$ and $g_1$ 
generate $F_1$ et $G_1$. So $af_1+bg_1\in A\subset F_{2n+1-i}$, which implies $bg_1\in F_{2n+1-i}$, hence $b=0$ or $i=0$. If $b=0$, then $V=F_1$, 
and consequently $F_1\subset G_1^{\perp}$, which is impossible by \textbf{(C9)}. So $i=0$. But then $af_1+bg_1\in G_3$, so $af_1\in G_3$ and also 
$a=0$. Hence $V=G_1\subset F_1^{\perp}$, which is excluded by \textbf{(C9)}.
 \item \label{b} Let $A=F_{2n-i}\cap G_{i+3}$. By \textbf{(C3)}, $\dim A=2$. By \textbf{(C5)}$_i$, $\dim A\cap G_1^{\perp}=1$, so 
$V=A\cap G_1^{\perp}$. Moreover $\dim V^{\perp}\cap F_2=1$. We have $W\supset V+G_1+V^{\perp}\cap F_2$. To determine $W$, it is enough to show that
 the sum is direct. First, $V+G_1$ is direct, because if it was not we would have $V=G_1$, so $G_1\subset F_{2n-i}$, which is impossible by 
\textbf{(C3)}. Finally the sum $V\oplus G_1+V^{\perp}\cap F_2$ is direct, or we would have $V^{\perp}\cap F_2\subset G_{i+3}$. But 
$\dim F_2\cap G_{i+3}=0$ by \textbf{(C3)} since $i\leq 2n-4$. 
 \item This case is similar to \ref{b} ; the proof uses \textbf{(C3)} and \textbf{(C6)}$_i$. 
 \item By \textbf{(C3)}, we get $\dim F_{2n-1}\cap G_{i+2}=1$, so $V=F_{2n-1}\cap G_{i+2}$. We must have $\dim W\cap F_2\neq 0$. But 
$V\not\subset F_2$, or else we would get $G_{i+2}\cap F_2\neq 0$, which is impossible by \textbf{(C3)} since $i\leq 2n-4$. Now 
$W\subset V^{\perp}$ implies $W\cap F_2\subset V^{\perp}\cap F_2$, which has dimension $1$ by \textbf{(C7)}$_i$. So 
$W\subset V^{\perp}\cap F_2\oplus V$. Similarly, using \textbf{(C8)}$_i$, we get $W\cap G_2=V^{\perp}\cap G_2$, so 
$W\supset V\oplus V^{\perp}\cap F_2 + V^{\perp}\cap G_2$. Now we only have to show that this sum is direct. If not, then there exists a non-zero 
vector of the form $av+bf_2$ in $V^{\perp}\cap G_2$, where $v$ and $f_2$ generate $V$ and $V^{\perp}\cap F_2$. As $v\in G_{i+2}$, we obtain 
$b f_2\in G_{i+2}$, so $b=0$ because $i\leq 2n-4$. Hence $V^{\perp}\cap G_2=V$ and consequently $V\subset G_2$ and $\dim F_{2n-i}\cap G_2 \leq 1$, 
which is impossible since $i\geq 2$. 
\end{enumerate}
The final formula for $I_1(\tau_{a,b},\tau_{c,d})$ follows from a straightforward calculation. \qedhere
\end{enumerate}
\end{proof}

\subsection{Computation of some invariants in $\Mod{3}{1}$}
\label{sec:mod.3}

In the previous section we computed the two-pointed invariants in $\IG$, which is equivalent to computing the quantum terms of the product by the 
hyperplane class $\tau_1$.
Indeed, the divisor axiom \cite[§ 2.2.4]{KM} yields 
\[
 I_1(\gamma_1,\gamma_2,\tau_1)=I_1(\gamma_1,\gamma_2),
\]
where $\gamma_1$ and $\gamma_2$ are any cohomology classes.
Hence to obtain a quantum Pieri rule for $\IGl$, we are left to compute the quantum product by $\tau_{1,1}$.
So we have to determine all invariants of the form $I_1(\tau_{1,1},\tau_{\lambda},\tau_{\mu})$ with $|\lambda|+|\mu|=6n-5$, that is to compute the 
number of lines through the following subvarieties :
\begin{align*}
 Y_1&=\left\{\Sigma\in\IG \mid \Sigma\cap F_{j+1}\neq 0,\Sigma\subset F_{2n+2-i-j}\right\}, \\
 Y_2&=\left\{\Sigma\in\IG \mid \Sigma\cap G_{l+1}\neq 0,\Sigma\subset G_{i+3-l}\right\}, \\
 Y_3&=\left\{\Sigma\in\IG \mid \Sigma\subset H\right\},
\end{align*}
where $0\leq i\leq 2n-1$, $0\leq 2j\leq 2n-1-i$ and $0\leq 2l\leq i$ are integers, $F_\bullet$ and $G_\bullet$ are isotropic flags and $H$ is a 
hyperplane. 

As before we use a genericity result which is proved in a similar way as Lemma \ref{lem:gen.2} :
\begin{lem}
\label{lem:gen.3}
 Assume $n\geq 2$.
 Then the set of $4$-uples $(F_{\bullet},G_{\bullet},H,\omega)\in\mathbb{F}_n\times\mathbb{F}_n\times\P{2n}\times\Lambda_n$ satisfying the following conditions
\begin{description}
 \item[(C1)] $\forall 0\leq p\leq 2n+1$, $\omega_{\mid F_p}$ has maximal rank ;
 \item[(C2)] $\forall 0\leq p\leq 2n+1$, $\omega_{\mid G_p}$ has maximal rank ;
 \item[(C3)] $\omega_{\mid H}$ is symplectic ;
 \item[(C4)] $\forall 0\leq p,q\leq 2n+1$, $F_p\cap G_q$ has the expected dimension ;
 \item[(C5)] $\forall 0\leq p,q\leq 2n+1$, $F_p\cap G_q\cap H$ has the expected dimension ;
 \item[(C6)$_i$] $\dim \left(F_{2n+2-i}\cap G_{i+3}\cap H\cap F_1^{\perp}\cap G_1^{\perp}\right)=1$ ; $\mathbf{(1\leq i\leq 2n-2)}$ ;
 \item[(C7)$_i$] $\dim F_{2n+1-i}\cap G_{i+3}\cap H\cap G_1^{\perp}=1$ and $\dim (F_{2n+1-i}\cap G_{i+3}\cap H\cap G_1^{\perp})^{\perp}\cap F_2=1$; $\mathbf{(0\leq i\leq 2n-3)}$ ;
 \item[(C8)$_i$] $\dim F_{2n+2-i}\cap G_{i+2}\cap H\cap F_1^{\perp}=1$ and $\dim (F_{2n+2-i}\cap G_{i+2}\cap H\cap F_1^{\perp})^{\perp}\cap G_2=1$; $\mathbf{(2\leq i\leq 2n-1)}$ ;
 \item[(C9)$_i$] $\dim \left(F_{2n+1-i}\cap G_{i+2}\cap H\right)^{\perp}\cap F_2=1$ ; $\mathbf{(2\leq i\leq 2n-3)}$ ;
 \item[(C10)$_i$] $\dim \left(F_{2n+1-i}\cap G_{i+2}\cap H\right)^{\perp}\cap G_2=1$ ; $\mathbf{(2\leq i\leq 2n-3)}$ ;
 \item[(C11)] $F_1\not\subset G_1^{\perp}$ ;
 \item[(C12)] $G_1\not\subset F_1^{\perp}$ ;
 \item[(C13)$_i$] $F_{2n-i}\cap G_{i+3}\cap H\cap G_1^{\perp}=0$ ; $\mathbf{(0\leq i\leq 2n-5)}$ ;
 \item[(C14)$_i$] $F_{2n+2-i}\cap G_{i+1}\cap H\cap F_1^{\perp}=0$ ; $\mathbf{(4\leq i\leq 2n-1)}$ ;
 \item[(C15)$_i$] $F_2\cap G_{i+3}\cap G_1^\perp=0$ ; $0\leq i\leq 2n-3$ ;
 \item[(C16)$_i$] $G_2\cap F_{2n+2-i}\cap F_1^\perp=0$ ; $2\leq i\leq 2n-1$ ;
\end{description}
is a dense open subset of $\mathbb{F}_n\times\mathbb{F}_n\times\P{2n}\times\Lambda_n$.
\end{lem}

Under these assumptions we can prove the

\begin{prop}
\label{prop:mod.3}
 \begin{enumerate}
\item The intersections $Y_i\cap\O$ are transverse. Moreover
  \begin{align*}
   Y_1\cap\O&=
     \begin{cases}
        \emptyset & \text{if $i+j\geq 2$} \\
        \left\{F_1\oplus K\right\}  & \text{if $i=1$ and $j=0$} \\
	\left\{K\oplus L\mid L\subset F_2\right\} &\text{if $i=0$ and $j=1$}
     \end{cases}
     \\
   Y_2\cap\O&=
       \begin{cases}
          \emptyset & \text{and $i\neq 2n-2$ or $l\neq 0$} \\
          \left\{G_1\oplus K\right\}  & \text{if $i=2n-2$ and $l=0$}
       \end{cases} \\
   Y_3\cap\O&=\emptyset.
   \end{align*}
 \item If $j$ ou $l\geq 2$, there is no line meeting $Y_1$, $Y_2$ and $Y_3$.
 \item If $j$ and $l\leq 1$, there is a unique line meeting $Y_1$, $Y_2$ and $Y_3$. Therefore the Gromov-Witten invariant $I_1(\tau_{1,1},\tau_{a,b},\tau_{c,d})$ equals $1$ if $a=c=2n-1$ and $b+c=2n-3$, and $0$ in the other cases.
\end{enumerate}
\end{prop}

\begin{proof}
 \begin{enumerate}
  \item The case of $Y_2\cap\O$ has already been treated in the proof of Proposition \ref{prop:mod.2}. If $\Sigma\in Y_1\cap\O$, we must have 
$K\subset F_{2n+2-i-j}$, so $i+j=1$.
If $i=1$ and $j=0$, then $Y_1\cap\O=\left\{K\oplus F_1\right\}$, and transversality is proven as in Proposition \ref{prop:mod.2}.
If $i=0$ and $j=1$, then $Y_1\cap\O=\left\{K\oplus L\mid L\subset F_2\right\}$. Take $\Sigma_0=K\oplus<f_2>$ where $f_2$ is a non-zero element 
in $F_2$. Again we express $\T{\Sigma_0}{Y_1}$ and $\T{\Sigma_0}{\O}$ as
subspaces of $\T{\Sigma_0}{G}$, where $G$ is the usual Grassmannian :
\begin{align*}
 \T{\Sigma_0}{Y_1}=\left\{\phi\in\T{\Sigma_0}{G}\mid\phi(f_2)\in F_2\slash <f_2>,\phi(k)\perp f_2\right\} \\
 \T{\Sigma_0}{\O}=\left\{\phi\in\T{\Sigma_0}{G}\mid\phi(k)=0\right\},
\end{align*}
with $k$ a generator of $K$. We see that the intersection of $\T{\Sigma_0}{Y_1}$ and $\T{\Sigma_0}{\O}$ has dimension $1$. 
Computing $\dim Y_1=2n-1$ and $\dim\O=2n-1$ we conclude that they generate $\T{\Sigma_0}{\IG}$.
Finally, $Y_3\cap\O=\emptyset$ since $K\not\subset H$ by \textbf{(C3)}.
  \item By \textbf{(C5)}, $F_{2n+2-i-j}\cap G_{i+3-l}\cap H=0$ as soon as $j+l\geq 3$. Moreover if $j=2$ and $l=0$ then we get $W\supset G_1$, 
hence $V\subset F_{2n-i}\cap G_{i+3}\cap H\cap G_1^{\perp}$. But this space is zero by \textbf{(C13)$_i$}, so there is no line. By 
\textbf{(C13)$_i$}, we get the same result when $j=0$ and $l=2$.
  \item There are four cases :
 \begin{enumerate}[a)]
  \item $j=l=0$ ;
  \item $j=1$, $l=0$ ;
  \item $j=0$, $l=1$ ;
  \item $j=l=1$.
 \end{enumerate}
 \begin{enumerate}[a)]
  \item We have $V=F_{2n+2-i}\cap G_{i+3}\cap H\cap F_1^{\perp}\cap G_1^{\perp}$ by \textbf{(C6)$_i$}. Moreover $W\supset V+F_1+G_1$. To obtain 
equality we only have to show that the sum is direct. First $V\neq F_1$ since $F_1\not\subset G_1^{\perp}$ by \textbf{(C11)}. Finally if 
$G_1\subset V\oplus F_1$, as $V\subset F_1^{\perp}$, we would have $G_1\subset F_1^{\perp}$, which is impossible by \textbf{(C12)}.
  \item We have $V=F_{2n+1-i}\cap G_{i+3}\cap H\cap G_1^{\perp}$ by \textbf{(C7)$_i$}. Moreover $W\subset V+G_1+F_2\cap V^{\perp}$. We prove now 
that this sum is direct. First $V\neq G_1$, or we would have $G_1\subset H$, which is excluded by \textbf{(C5)}. Now 
$F_2\cap V^\perp\not\subset V\oplus G_1$ since $F_2\cap G_{i+3}\cap G_1^\perp=0$ for $i\leq 2n-3$ by \textbf{(C15)$_i$}.
  \item $V=F_{2n+2-i}\cap G_{i+2}\cap H\cap F_1^\perp$ by \textbf{(C8)$_i$}. Moreover $W\supset V+F_1+G_2\cap V^\perp$ (by \textbf{(C9)$_i$} and 
\textbf{(C10)$_i$}), and this sum is direct (same argument than in the previous case, using condition \textbf{(C16)$_i$}).
  \item $V=F_{2n+1-i}\cap G_{i+2}\cap H$, $W\supset V+W\cap F_2+W\cap G_2=V+F_2\cap V^\perp+G_2\cap V^\perp$. This sum is direct ; indeed, 
$F_2\cap V^\perp\neq G_2\cap V^\perp$ car $F_2\cap G_2=0$ by \textbf{(C4)} ; in addition $V\not\subset F_2\cap V^\perp\oplus G_2\cap V^\perp$, or 
we would get $G_2\cap F_{2n+1-i}\neq 0$, which is impossible by $i\geq 2$.
 \end{enumerate}
The final formula for $I_1(\tau_{1,1},\tau_{a,b},\tau_{c,d})$ follows from a straightforward calculation. \qedhere
\end{enumerate}
\end{proof}

\subsection{Quantum Pieri rule}
\label{sec:qpieri}

We can now prove Theorem \ref{theo:qpieri} :

\begin{proof}[Proof of Theorem \ref{theo:qpieri}]
We start with the invariants $I_1(\tau_1,\tau_{a,b},\tau_{c,d})$, which are equal to the two-pointed invariants $I_1(\tau_{a,b},\tau_{c,d})$ because of
the divisor axiom. The first item of Proposition \ref{prop:mod.2} enables us to apply the Enumerativity theorem \ref{theo:enum}. Then we use the 
second item of Proposition \ref{prop:mod.2}. For $j=l=0$ we get that for all $0\leq i\leq 2n-2$, we have $I_1(\tau_{2n-1,i},\tau_{2n-1,2n-2-i})=1$.
Then setting $j=0$ and $l>0$ we recursively get $I_1(\tau_{2n-1,i},\tau_{2n-1-l,2n-2-i+l})=0$ (for all $i$ and $l>0$). Finally, setting $j$ 
and $l>0$ we get $I_1(\tau_{2n-1-j,i+j},\tau_{2n-1-l,2n-2-i+l})=0$ (for all $i$ and $j,l>0$).
Hence :
\begin{align*}
 I_1(\tau_1,\tau_{a,b},\tau_{c,d}) &= \begin{cases}
	1 \text{ if $a=c=2n-1$,} \\
	0 \text{ if $a$ or $c<2n-1$}.
     \end{cases}
\end{align*}
Similarly, Proposition \ref{prop:mod.3} and Theorem \ref{theo:enum} imply
\begin{align*}
 I_1(\tau_{1,1},\tau_{a,b},\tau_{c,d}) &= \begin{cases}
	1 \text{ if $a=c=2n-1$,} \\
	0 \text{ if $a$ or $c<2n-1$}.
     \end{cases}
\end{align*}
Using the classical Pieri rule and Poincar\'e duality, we get our result.
\end{proof}

Using the quantum Pieri formula we can fill out the Hasse diagram from Figure \ref{fig:hasse.2.7} to obtain the quantum Hasse diagram of 
$\mathrm{IG}(2,7)$ in Figure \ref{fig:qhasse.2.7}. As a comparison see the quantum Hasse diagram of $\mathrm{IG}(2,6)$ in Figure 
\ref{fig:qhasse.2.6}.

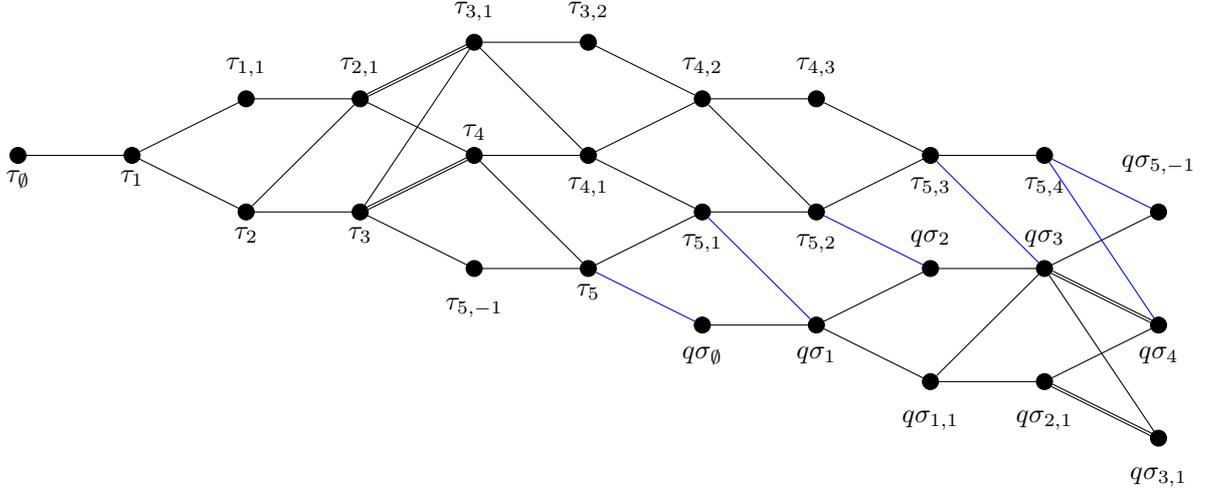
\begin{figure}[h!]
 \centering
 \begin{tikzpicture}[scale=1.5]
\tikzstyle{every node}=[draw,circle,fill=black,minimum size=6pt,inner sep=0pt]

%classical Hasse diagram

\draw (0,0) node (0) [label=below:$\tau_{\emptyset}$] {}
        -- (1,0) node (1) [label=below:$\tau_1$] {}
        -- (2,-0.5) node (2) [label=below:$\tau_{2}$] {}
        -- (3,-0.5) node (3) [label=below:$\tau_{3}$] {}
        -- (4,-1) node (5-1) [label=below:$\tau_{5,-1}$] {}
        -- (5,-1) node (5) [label=below:$\tau_{5}$] {}
        -- (6,-0.5) node (51) [label=below:$\tau_{5,1}$] {}
        -- (7,-0.5) node (52) [label=below:$\tau_{5,2}$] {}
        -- (8,0) node (53) [label=below:$\tau_{5,3}$] {}
        -- (9,0) node (54) [label=below:$\tau_{5,4}$] {};
\draw (1) 
	-- (2,0.5) node (11) [label=above:$\tau_{1,1}$] {}
	-- (3,0.5) node (21) [label=above:$\tau_{2,1}$] {}
	-- (4,0) node (4) [label=above:$\tau_{4}$] {}
	-- (5,0) node (41) [label=below:$\tau_{4,1}$] {}
	-- (6,0.5) node (42) [label=above:$\tau_{4,2}$] {}
	-- (7,0.5) node (43) [label=above:$\tau_{4,3}$] {}
	-- (53);
\draw[double] 
	(3) -- (4)
	(21) -- (4,1) node (31) [label=above:$\tau_{3,1}$] {};
\draw (3) 
	-- (31)
	-- (5,1) node (32) [label=above:$\tau_{3,2}$] {}
	-- (42)
	-- (52);
\draw (2) -- (21);
\draw (4) -- (5);
\draw (31) -- (41) -- (51);

%beginning of the classical Hasse diagram (shifted)

\draw (6,-1.5) node (q0) [label=below:$q\sigma_{\emptyset}$] {}
        -- (7,-1.5) node (q1) [label=below:$q\sigma_1$] {}
        -- (8,-1) node (q2) [label=above:$q\sigma_{2}$] {}
        -- (9,-1) node (q3) [label=above:$q\sigma_{3}$] {}
        -- (10,-0.5) node (q5-1) [label=above:$q\sigma_{5,-1}$] {}
;
\draw (q1)
	-- (8,-2) node (q11) [label=below:$q\sigma_{1,1}$] {}
	-- (9,-2) node (q21) [label=below:$q\sigma_{2,1}$] {}
	-- (10,-1.5) node (q4) [label=below:$q\sigma_{4}$] {}
;
\draw (q11) -- (q3) ;
\draw[double]
	(q3) -- (q4)
	(q21) -- (10,-2.5) node (q31) [label=below:$q\sigma_{3,1}$] {}
;
\draw (q3) -- (q31);

%quantum corrections

\draw[blue]
	(5) -- (q0)
	(51) -- (q1)
	(52) -- (q2)
	(53) -- (q3)
	(54) -- (q4)
	(54) -- (q5-1);

\end{tikzpicture}

 \caption{Quantum Hasse diagram of $\mathrm{IG}(2,7)$}
 \label{fig:qhasse.2.7}
\end{figure}

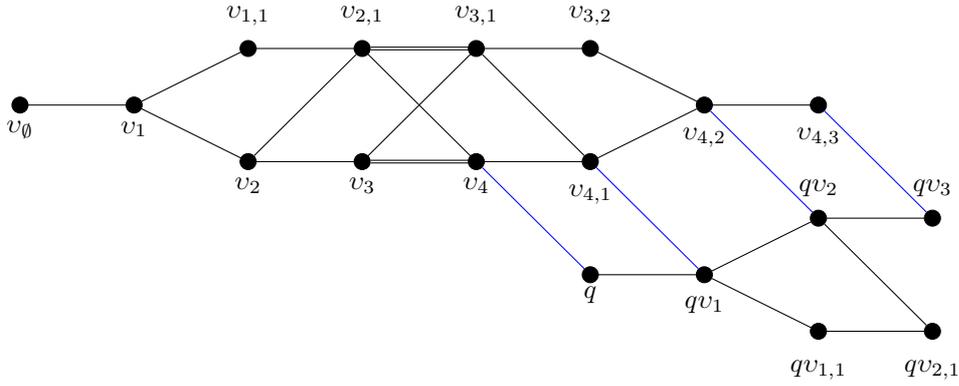
\begin{figure}
 \centering
 \begin{tikzpicture}[scale=1.5]
 
\tikzstyle{every node}=[draw,circle,fill=black,minimum size=6pt,inner sep=0pt]

\draw (0,0) node (0) [label=below:$\upsilon_{\emptyset}$] {}
        -- (1,0) node (1) [label=below:$\upsilon_1$] {}
        -- (2,-0.5) node (2) [label=below:$\upsilon_{2}$] {}
        -- (3,-0.5) node (3) [label=below:$\upsilon_{3}$] {}
        -- (4,0.5) node (31) [label=above:$\upsilon_{3,1}$] {}
        -- (5,0.5) node (32) [label=above:$\upsilon_{3,2}$] {}
        -- (6,0) node (42) [label=below:$\upsilon_{4,2}$] {}
        -- (7,0) node (43) [label=below:$\upsilon_{4,3}$] {};
\draw (1) 
	-- (2,0.5) node (11) [label=above:$\upsilon_{1,1}$] {}
	-- (3,0.5) node (21) [label=above:$\upsilon_{2,1}$] {}
	-- (4,-0.5) node (4) [label=below:$\upsilon_{4}$] {}
	-- (5,-0.5) node (41) [label=below:$\upsilon_{4,1}$] {}
	-- (42);
\draw[double] 
	(21) -- (31)
	(3) -- (4);
\draw 
	(2) --(21)
	(31) -- (41);
\draw (5,-1.5) node (q) [label=below:$q$] {}
	-- (6,-1.5) node (q1) [label=below:$q\upsilon_1$] {}
	-- (7,-1) node (q2) [label=above:$q\upsilon_2$] {}
	-- (8,-1) node (q3) [label=above:$q\upsilon_3$] {}
    (q1)-- (7,-2) node (q11) [label=below:$q\upsilon_{1,1}$] {}
	-- (8,-2) node (q21) [label=below:$q\upsilon_{2,1}$] {}
    (q2)-- (q21);
\draw[blue] (4) -- (q)
	    (41) -- (q1)
	    (42) -- (q2)
	    (43) -- (q3);
\end{tikzpicture}

 \caption{Quantum Hasse diagram of $\mathrm{IG}(2,6)$}
 \label{fig:qhasse.2.6}
\end{figure}

\subsection{Quantum presentation}
\label{sec:qpres}

\begin{prop}[Presentation of $\mathrm{QH}^*\left(\IGl,\Z\right)$]
\label{prop:qpres}
The ring $\mathrm{QH}^*\left(\IGl,\Z\right)$ is ge\-nerated by the classes $\tau_1$, $\tau_{1,1}$ and the quantum parameter $q$. The relations are
\begin{align*}
\det\left(\tau_{1^{1+j-i}}\right)_{1\leq i,j\leq 2n} &= 0,\\
\frac{1}{\tau_1}\det\left(\tau_{1^{1+j-i}}\right)_{1\leq i,j\leq 2n+1}+q &= 0
\end{align*}
\end{prop}

\begin{proof}
 Siebert and Tian proved in \cite[Prop 2.2]{ST} that the quantum relations are obtained by evaluating the classical relations using the quantum product.
Define $\delta_{2n}$ and $\delta_{2n+1}'$ as in the proof of Proposition \ref{prop:pres.sympl} and denote by $\overline{\delta_{2n}}$ and 
$\overline{\delta_{2n+1}'}$ the same expressions with the cup product replaced by the quantum product. 

Now we consider the quantum products $\Pi_a:=(\tau_1)^{2(n-a)}\star(\tau_{1,1})^a$ for $0\leq a\leq n$. For reasons of degree it has no $q$-term 
of degree greater than $1$.
First we prove that $\Pi_a$ has no $q$-term if $a\neq 0,1$. To prove this, we decompose $\Pi_a$ for $a>0$ as 
\[
 \Pi_a=\tau_{1,1}\star\left((\tau_1)^{2(n-a)}(\tau_{1,1})^{a-1}\right).
\]
Notice that for degree reasons, $(\tau_1)^{2(n-a)}(\tau_{1,1})^{a-1}$ has no $q$-term. Moreover, if $a\geq 2$, the classical Pieri formula 
\ref{prop:pieri.1} implies that this product contains only classes $\tau_{c,d}$ with $c<2n-1$. Then we use the quantum Pieri formula 
\ref{theo:qpieri} to conclude that there is no $q$-term in $\Pi_a$. 
We are now left with computing the $q$-term of $\Pi_0$ and $\Pi_1$. Set $\alpha_p:=(\tau_1)^p$ for $p\leq 2n-1$. $\alpha_p$ has no $q$-term. 
We have $\Pi_0=\tau_1\star\alpha_{2n-1}$ and $\Pi_1=\tau_{1,1}\star\alpha_{2n-2}$. We compute recursively the coefficients of $\tau_p$ and 
$\tau_{p-1,1}$ for $p\leq 2n-3$ in $\alpha_p$ using the classical Pieri rule. We find
\[
 \alpha_p=\tau_p+(p-1)\tau_{p-1,1}+\text{terms with lower first part}.
\]
Then 
\[
 \alpha_{2n-2}=\tau_{2n-1,-1}+(2n-2)\tau_{2n-2}+\text{terms with lower first part}
\]
and
\[
 \alpha_{2n-1}=(2n-1)\tau_{2n-1}+\text{terms with lower first part}.
\]
Finally we use the quantum Pieri rule to deduce that
\begin{align*}
 \Pi_0=&\text{classical terms}+(2n-1)q, \\
 \Pi_1=&\text{classical terms}+q.
\end{align*}
But
\begin{align*}
 \overline{\delta_{2n}}&=\Pi_0-(2n-1)\Pi_1+\text{linear combination of $\Pi_a$'s with $a\geq 2$,} \\
 \overline{\delta_{2n+1}'}&=\Pi_0-2n\Pi_1+\text{linear combination of $\Pi_a$'s with $a\geq 2$},
\end{align*}
hence $\overline{\delta_{2n}}=\delta_{2n}$ and $\overline{\delta_{2n+1}'}=\delta_{2n+1}'-q$.
\end{proof}

\subsection{Around a conjecture of Dubrovin}
\label{sec:dubrovin}

In 1994, Dubrovin stated a conjecture relating properties of the quantum cohomology of Fano varieties and properties of their derived category :
\begin{conj}[Dubrovin {\cite[Conj 4.2.2]{dubrovin}}]
 Let $X$ be a Fano variety. The big quantum cohomology of $X$ is generically semisimple if and only if its derived category of coherent sheaves $\db{X}$ admits a full exceptional collection.
\end{conj}
Here we check this conjecture for odd symplectic Grassmannians of lines. We first show that the (small) quantum cohomology ring of $\IGl$, localized at $q\neq 0$, is semisimple. To do this we adapt the presentation of Proposition \ref{prop:qpres} to make the symmetries more apparent :

\begin{theo}
 \label{theo:semisimple}
\begin{enumerate}
 \item The ring $\mathrm{QH}^*\left(\IGl,\Z\right)$ is isomorphic to $R^{\mathfrak{S}_2}$, where
\[
 R=\Z\left[x_1,x_2,q\right]\slash \left(h_{2n}(x_1,x_2),h_n(x_1^2,x_2^2)+q\right)
\]
and $x_1$ and $x_2$ are the Chern roots of the tautological bundle $\mathrm{S}$ and $h_r(y_1,y_2)$ is the $r$-th complete symmetric 
function of the variables $y_1,y_2$.
 \item $\mathrm{QH}^*\left(\IGl,\Z\right)_{q\neq 0}$ is semisimple.
\end{enumerate}
\end{theo}

\begin{proof}
\begin{enumerate}
 \item  We use the recurrence relation \eqref{eq:rec.delta} from Proposition \ref{prop:pres.sympl} to prove that $\delta_r=h_r(x_1,x_2)$ for all $r$.
 Then 
\[
  \delta_{2n+1}'=\frac{h_{2n+1}(x_1,x_2)}{x_1+x_2}=h_{n}(x_1^2,x_2^2).
\]
 \item It is enough to prove the semisimplicity of $R$ localized at $q\neq 0$. We may assume $q=-1$. Using 
$(x_1-x_2)h_{2n}(x_1,x_2)=x_1^{2n+1}-x_2^{2n+1}$ and noticing that we must have $x_2\neq 0$, the first relation implies that $x_1=\zeta x_2$, 
where $\zeta\neq 1$ is a $(2n+1)$-th root of unity. Replacing in the second relation $h_n(x_1^2,x_2^2)-1=0$, we get $x_1^{2n}=1+\zeta$. Since 
$\zeta\neq -1$, this equation has $2n$ distinct solutions. So we have $2n$ distinct solutions for $x_1$, and for each $x_1$ we have $2n$ distinct 
solutions for $x_2$, which gives us (at least) $4n^2$ distinct solutions for the pair $(x_1,x_2)$. But the number of solutions, counted with their 
multiplicity, should be equal to twice the rank of $\mathrm{H}^*\left(\IGl,\Z\right)$, which is equal to $2n^2$. So there are no other solutions, 
and all solutions are simple. Hence the semisimplicity. \qedhere
\end{enumerate}
\end{proof}
Since semisimplicity of the small quantum cohomology implies generic semisimplicity of the big one, Theorem \ref{theo:semisimple} proves that the big quantum cohomology of $\IGl$ is generically semisimple. So to confirm Dubrovin's conjecture in this case it is enough to find a full exceptional collection. 

In \cite{kuznetsov}, Kuznetsov computed full exceptional collections for the symplectic Grassmannian of lines. His method can easily be adapted to the odd symplectic case. Here we denote the tautological bundle by $\UU$ instead of $\SS$ to avoid confusion with symmetric powers. We denote by $\YY_{2n+1}^I$ the following collection of integer pairs, ordered lexicographically :
\[
 \YY_{2n+1}^I := \left\{ (k,l) \in \Z^2 \mid 0 \leq k \leq 2n-1 \text{ et } 0 \leq l \leq n-1 \right\}.
\]
\begin{theo}
 The collection $\mathcal{C}^I := \left\{ \SS^l\UU^*(k) \mid (k,l) \in \YY_{2n+1}^I \right\}$ is a full exceptional collection in $\db{\IG(2,2n+1)}$.
\end{theo}

\begin{proof}
 We use the Lefschetz full exceptional collection for the Grassmannian $\G(2,2n+1)$ introduced by Kuznetsov. Let
\[
 \YY_{2n+1} := \left\{ (k,l) \in \Z^2 \mid 0 \leq k \leq 2n \text{ et } 0 \leq l \leq n-1 \right\}.
\]
Then $\mathcal{C} := \left\{ \SS^l\UU^*(k) \mid (k,l) \in \YY_{2n+1} \right\}$ is a full exceptional collection in $\db{\G(2,2n+1)}$ (cf \cite[Thm 4.1]{kuznetsov}). Since $\mathcal{C}$ is a Lefschetz exceptional collection, it follows from \cite[Prop 2.4]{kuznetsov} that $\mathcal{C}^I$ is an exceptional collection for $\IG(2,2n+1)$. As in the even case, it remains to show by induction that it is full. We first introduce
\[
\widetilde{\YY}_{2n-1}^I = \left\{ (k,l) \in \Z^2 \mid 0 \leq k \leq 2n-1 \text{ and } 0 \leq l \leq n \right\}
\]
and prove the
\begin{lem}
\label{lem:exc.pair.1}
 For all $(k,l) \in \widetilde{\YY}_{2n-1}^I$, the vector bundle $S^l\UU^*(k)$ lies in the subcategory of the derived category $\db{\IG(2,2n+1)}$ of $\IGl$ generated by the Lefschetz collection $\mathcal{C}^I$.
\end{lem}

\begin{proof}[Proof of the lemma]
 We have $\widetilde{\YY}_{2n-1}^I\setminus \YY_{2n-1}=\left\{(0,n),\dots,(2n-1,n)\right\}$. Moreover, we have the following exact sequences on $\G(2,2n+1)$ for $0\leq p\leq n$ from \cite[Equation (11)]{kuznetsov} :
\begin{align*}
 0 & \ra S^{2n-1-p}\UU^* \ra W \otimes S^{2n-2-p}\UU^*(1) \ra \dots \ra \bigwedge^{2n-2-p}W \otimes \UU^*(2n-2-p) \ra \\
   & \ra \bigwedge^{2n-1-p}W \otimes \UU^*(2n-1-p) \ra \bigwedge^p W^* \otimes \OO(2n-p) \ra \bigwedge^{p-1} W^* \otimes \UU^*(2n-p) \\
   & \ra \dots \ra S^p\UU^*(2n-p) \ra 0.
\end{align*}
These exact sequences can be restricted to $\IG(2,2n+1)$. 

For $p=n$, we get a resolution of $S^n\UU^*(n)$ by objects of the subcategory generated by $\mathcal{C}$. Tensoring by $\OO(i)$ for $1 \leq i \leq n-1$, we also get resolutions for $S^n\UU^*(n+1),\dots,S^n\UU^*(2n-1)$. Then for $p=n-1$, we get a resolution of $S^n\UU^*$, and again, after tensoring by $\OO(i)$ for $1 \leq i \leq n-2$, we obtain resolutions of $S^n\UU^*(1),\dots,S^n\UU^*(n-2)$. 
To conclude the proof, we are left with finding a resolution of $S^n\UU^*(n-1)$ by objects of the subcategory generated by $\mathcal{C}$. Let $\overline{W}\supset W$ be a $(2n+2)$-dimensional vector space endowed with a symplectic form $\overline{\omega}$ extending $\omega$. On $\G_{\overline{\omega}}(2,\overline{W})$, we have the following bicomplex from \cite[Prop 5.3]{kuznetsov} :

%\hspace{-1.3cm}
 \begin{tikzpicture}[scale=1.5]
 
\tikzstyle{every node}=[draw,circle,fill=black,minimum size=6pt,inner sep=0pt]

\draw (0,0) node (0) [label=below:$\upsilon_{\emptyset}$] {}
        -- (1,0) node (1) [label=below:$\upsilon_1$] {}
        -- (2,-0.5) node (2) [label=below:$\upsilon_{2}$] {}
        -- (3,-0.5) node (3) [label=below:$\upsilon_{3}$] {}
        -- (4,0.5) node (31) [label=above:$\upsilon_{3,1}$] {}
        -- (5,0.5) node (32) [label=above:$\upsilon_{3,2}$] {}
        -- (6,0) node (42) [label=below:$\upsilon_{4,2}$] {}
        -- (7,0) node (43) [label=below:$\upsilon_{4,3}$] {};
\draw (1) 
	-- (2,0.5) node (11) [label=above:$\upsilon_{1,1}$] {}
	-- (3,0.5) node (21) [label=above:$\upsilon_{2,1}$] {}
	-- (4,-0.5) node (4) [label=below:$\upsilon_{4}$] {}
	-- (5,-0.5) node (41) [label=below:$\upsilon_{4,1}$] {}
	-- (42);
\draw[double] 
	(21) -- (31)
	(3) -- (4);
\draw 
	(2) --(21)
	(31) -- (41);
\draw (5,-1.5) node (q) [label=below:$q$] {}
	-- (6,-1.5) node (q1) [label=below:$q\upsilon_1$] {}
	-- (7,-1) node (q2) [label=above:$q\upsilon_2$] {}
	-- (8,-1) node (q3) [label=above:$q\upsilon_3$] {}
    (q1)-- (7,-2) node (q11) [label=below:$q\upsilon_{1,1}$] {}
	-- (8,-2) node (q21) [label=below:$q\upsilon_{2,1}$] {}
    (q2)-- (q21);
\draw[blue] (4) -- (q)
	    (41) -- (q1)
	    (42) -- (q2)
	    (43) -- (q3);
\end{tikzpicture}

whose associated total complex is exact. Restricting this complex to
\[
 \G_\omega(2,W)\subset\G_{\overline{\omega}}(2,\overline{W}),
\]
we get a complex such that every entry except $S^n\UU^*(n-1)$ is an object of the subcategory generated by $\mathcal{C}$. Hence $S^n\UU^*(n-1)$ also belongs to the latter category.
\end{proof}
We will need a further lemma for the inductive step :
\begin{lem}
\label{lem:exc.pair.2}
 Let $V\subset W$ be a $(2n-1)$-dimensional vector space such that $\omega_{\mid V}$ has maximal rank. Let $X:=\IG(2,W)$, $X_V:=\IG(2,V)$ and $i_V : X_V \hookrightarrow X$ be the natural embedding. We get the following Koszul resolution for ${i_V}_*\OO_{X_V}$~:
\begin{align*}
 0 & \ra \OO_{X}(-2) \ra \UU_X(-1) \oplus \UU_X(-1) \ra \OO_{X}(-1)^{\oplus 3} \oplus S^2\UU_X \ra \UU_X \oplus \UU_X \ra \\
   & \ra \OO_{X} \ra {i_V}_*\OO_{X_V} \ra 0.
\end{align*}
\end{lem}
\begin{proof}[Proof of the lemma]
To each vector space $V$ as above corresponds a section $\phi_V$ of $\UU^*\oplus\UU^*$ on $\IG(2,W)$ ; moreover, the zero locus of $\phi_V$ is $\IG(2,V)\subset\IG(2,W)$. Since
\[
 \dim X_V=4n-7=\dim X-4=\mathrm{rg}( \UU_X^* \oplus \UU_X^* ),
\]
any such section $\phi_V$ is regular, hence the sheaf $i_*\OO_{X_V}$ admits a Koszul resolution of the above form.
\end{proof}
We may now use induction on $n$ to prove the theorem. For $n=1$, the result is obvious. Now assume that $n\geq 2$, that the result is proved for $n-1$, and that the Lefschetz collection for $n$ is not full. Then there exists an object $F\in\db{\IGl}$ which is right orthogonal to all bundles in the collection, \emph{i.e} such that :
\[
 0 = \mathrm{RHom}( S^l\UU^*(k) , F ) = \H^\bullet ( X , S^l\UU^*(-k) \otimes F )
\]
for all $(k,l)\in\widetilde{\YY}_{2n-1}^I$. Let $V$ be such as in Lemma \ref{lem:exc.pair.2} and $i_V : X_V \hookrightarrow X$ be the embedding. We will prove $i_V^* F=0$. Let $(k,l)\in\widetilde{\YY}_{2n-1}^I$. Tensoring the resolution of Lemma \ref{lem:exc.pair.2} by $S^l\UU^*(-k) \otimes F$, we get
\begin{align*}
 0 & \ra S^l\UU(-2-k) \otimes F \ra \left( S^{l+1}\UU(-1-k) \otimes F \oplus S^{l-1}\UU(-2-k) \otimes F \right)^{\oplus 2} \ra \\
   & \ra \left( S^l\UU(-1-k) \otimes F \right)^{\oplus 4} \oplus S^{l+2}\UU(-k) \otimes F \oplus S^{l-2} \UU (-2-k) \otimes F \ra \\
   & \ra \left( S^{l+1}\UU(-k) \otimes F \oplus S^{l-1}\UU(-1-k) \otimes F\right)^{\oplus 2} \ra S^l\UU(-k) \otimes F \ra \\
   & \ra S^l\UU(-k) \otimes F \otimes {i_V}_*\OO_{X_V} \ra 0.
\end{align*}
Moreover
\[
 S^l\UU(-k) \otimes F \otimes {i_V}_*\OO_{X_V} \cong {i_V}_* \left( {i_V}^* (S^l\UU(-k) \otimes F) \right) \cong {i_V}_* \left( S^l\UU(-k) \otimes {i_V}^* (F) \right).
\]
If $(k,l)\in\YY_{2n-1}^I$, then $(k+2,l)$, $(k+2,l-1)$, $(k+1,l+1)$, $(k+1,l)$, $(k+2,l-2)$, $(k,l+2)$, $(k,l+1)$, $(k+1,l-1)$ and $(k,l)$ are in $\widetilde{\YY}_{2n-1}^I$. Hence the cohomology of the five first terms of the above complex vanishes, and
\[
 \mathrm{RHom}_V( S^l\UU^*(k) , {i_V}^* F ) = \H^\bullet ( X_V , S^l\UU^*(-k) \otimes {i_V}^* F ) = 0 
\]
for all $(k,l) \in\YY_{2n-1}^I$. By the induction hypothesis, we get ${i_V}^* F = 0$.

\begin{lem}
\label{lem:annul.cat}
 If, for some $F\in\db{\IGl}$, we have ${i_V}^* F = 0$ for every $(2n-1)$-dimensional vector space $V$ for which the conditions of Lemma \ref{lem:exc.pair.2} hold, then $F=0$.
\end{lem}

\begin{proof}[Proof of lemma]
 If $F \neq 0$, let $q$ be the maximal integer such that $\H^q(F)\neq 0$. Let $P\in\mathrm{Supp} \H^q(F)$. Then as long as $n\geq 2$, there exists a $(2n-1)$-dimensional subspace $V$ such that $P\subset V$ and $\omega_{\mid V}$ has maximal rank. The functor $i_V^*$ being left exact, we obtain that $\H^q(i_V^* F)\neq 0$, hence $i_V^* F\neq 0$.
\end{proof}
This concludes the proof of the theorem.
\end{proof}

It should be mentioned that it is not known whether the Dubrovin conjecture holds for the symplectic Grassmannian of lines. Indeed, although Kuznetsov has found a full exceptional collection for these varieties, Chaput and Perrin proved in \cite[Thm. 4]{CP} that their small quantum cohomology is not semisimple. What happens for the big quantum cohomology is still unknown.

% Bibliographie

\bibliographystyle{alpha}
\bibliography{biblio}

\end{document}